\documentclass[a4paper,12pt]{amsart}
\usepackage{graphicx,latexsym,graphics}
\usepackage{amsfonts, amssymb, amsmath, amsthm}
\usepackage{algorithmic, algorithm}

\newcommand{\ds}{\displaystyle}

\newcommand{\NN}{\mathbb N}

\newcommand{\CC}{\mathbb C}
\newcommand{\RR}{\mathbb R}
\newcommand{\ZZ}{\mathbb Z}
\newcommand{\EE}{\mathcal E}
\newcommand{\DD}{\mathcal D}
\newcommand{\SSS}{\mathcal S}

\theoremstyle{plain}
\newtheorem{theorem}{Theorem}[section]
\newtheorem{proposition}[theorem]{Proposition}
\newtheorem{lemma}[theorem]{Lemma}
\newtheorem{corollary}[theorem]{Corollary}

\theoremstyle{remark}
\newtheorem{remark}[theorem]{Remark}

\theoremstyle{definition}
\newtheorem{definition}[theorem]{Definition}
\newtheorem{example}[theorem]{Example}

\numberwithin{equation}{section}

\allowdisplaybreaks

\newcommand{\beq}{\begin{eqnarray}}
\newcommand{\eeq}{\end{eqnarray}}

\newcommand{\beqs}{\begin{eqnarray*}}
\newcommand{\eeqs}{\end{eqnarray*}}

\begin{document}

\title[Translation-invariant spaces of quasianalytic ultradistributions]{On a class of translation-invariant spaces of quasianalytic ultradistributions}

\author[P. Dimovski]{Pavel Dimovski}
\address{P. Dimovski, Faculty of Technology and Metallurgy, University Ss. Cyril and Methodius, Ruger Boskovic 16, 1000 Skopje, Macedonia}
\email{dimovski.pavel@gmail.com}

\author[B. Prangoski]{Bojan Prangoski}
\address{B. Prangoski, Faculty of Mechanical Engineering, University Ss. Cyril and Methodius,
Karpos II bb, 1000 Skopje, Macedonia}
\email{bprangoski@yahoo.com}

\author[J. Vindas]{Jasson Vindas}
\thanks{J. Vindas gratefully acknowledges support by Ghent University, through the BOF-grant 01N01014.}
\address{J. Vindas, Department of Mathematics, Ghent University, Krijgslaan 281 Gebouw S22, 9000 Gent, Belgium}
\email{jvindas@cage.UGent.be}

\subjclass[2010]{Primary 46F05. Secondary 46H25; 46F10;  46F15; 46E10}
\keywords{Quasianalytic ultradistributions; Convolution of ultradistributions; Translation-invariant Banach space of ultradistributions; Tempered ultradistributions; Beurling algebra; Hyperfunctions}

\maketitle

\begin{center}{\em Dedicated to Professor B. Stankovi\'c on the occasion of his 90th birthday and to Professor J. Vickers on the occasion of his 60th birthday}\end{center}

\begin{abstract}
A class of translation-invariant Banach spaces of quasianalytic ultradistributions is introduced and studied. They are Banach modules over a Beurling algebra. Based on this class of Banach spaces, we define corresponding test function spaces $\mathcal{D}^*_E$ and their strong duals $\mathcal{D}'^*_{E'_{\ast}}$ of quasianalytic type, and study convolution and multiplicative products on $\mathcal{D}'^*_{E'_{\ast}}$. These new spaces generalize previous works about translation-invariant spaces of tempered (non-quasianalytic ultra-) distributions; in particular, our new considerations apply to the settings of Fourier hyperfunctions and ultrahyperfunctions. New weighted $\mathcal{D}'^{\ast}_{L^{p}_{\eta}}$ spaces of quasianalytic ultradistributions are analyzed.
\end{abstract}
\maketitle

\section{Introduction}

Recently, the authors and Pilipovi\'{c} have constructed and studied new classes of distribution and non-quasianalytic ultradistribution spaces in connection with translation-invariant Banach spaces \cite{DPV,DPPV}. Those spaces generalize the concrete instances of weighted $\mathcal{D}'_{L^{p}}$ and $\mathcal{D}'^{*}_{L^{p}}$ spaces \cite{PilipovicK,S} and have shown usefulness in the study of boundary values of holomorphic functions \cite{DPVBV} and the convolution of generalized functions \cite{DPPV}.

The aim of this article is to extend the theory of ultradistribution spaces associated to translation-invariant Banach spaces by considering mixed quasianalytic cases. We have been able here to transfer all results from \cite{DPPV} to this new setting with the aid of
various new important results for quasianalytic ultradistribution spaces of type $\mathcal{S}'^{\ast}_{\dagger}(\mathbb{R}^{d})$ (see Subsection \ref{notation} for the notation) from \cite{ppv} concerning the construction of parametrices and the structure of these spaces. Such technical results will be stated in Section \ref{PPV} without proofs, as details will be treated in \cite{ppv}. Although our results in the present paper are analogous to those from \cite{DPPV}, new arguments and ideas have had to be developed here in order to deal with the quasianalytic case and achieve their proofs.

In Section \ref{TIBU} we study the class of translation-invariant Banach spaces of ultradistributions of class $*-\dagger$. These are translation-invariant Banach spaces satisfying $\mathcal{S}^*_\dagger(\mathbb{R}^{d})\hookrightarrow E \hookrightarrow \mathcal{S}'^*_\dagger(\mathbb{R}^{d})$ and having ultrapolynomially bounded weight function of class $\dagger$. Here $\ast$ and $\dagger$ stand for the Beurling and Roumieu cases of sequences $M_{p}$ and $A_{p}$, respectively. We would like emphasize that our considerations apply to hyperfunctions and ultra-hyperfunctions, which correspond to the symmetric choices $M_{p}=A_p=p!$; but more generally, our weight sequence $M_{p}$, measuring the ultradifferentiability, is allowed to satisfy the mild condition $p!^{\lambda}\subset M_p$ with $\lambda>0$. The growth assumption on $A_{p}$ is just $p!\subset A_p$, which also allows to deal with Banach spaces whose translation groups may have exponential growth.

Section \ref{new spaces} contains our main results. In analogy to \cite{DPPV}, we introduce the test function spaces $\mathcal{D}_E^{(M_p)}$, $\mathcal{D}_E^{\{M_p\}}$, and $\tilde{\mathcal{D}}_E^{\{M_p\}}$. We prove that the following continuous and dense embeddings hold  $\mathcal{S}^*_\dagger(\mathbb{R}^d)\hookrightarrow\mathcal{D}^{*}_E\hookrightarrow E\hookrightarrow \mathcal{S}'^*_\dagger(\mathbb{R}^d)$ and that $\mathcal{D}^{*}_E$ are topological modules over the Beurling algebra
$L^{1}_{\omega}$, where $\omega$ is the weight function of the translation group of $E$. We also prove the dense embedding $\mathcal{D}^*_E\hookrightarrow\mathcal{O}^*_{\dagger,C}(\mathbb{R}^{d})$, where the spaces $\mathcal{O}^*_{\dagger,C}(\mathbb{R}^{d})$ are defined in a similar way as in \cite{DPPV}.
The space $\mathcal D'^*_{E'_*}$ is defined as the strong dual of $\mathcal{D}_{E}$ and various structural and topological properties of $\mathcal{D}'^*_{E'_{\ast}}$ are obtained via the parametrix method (Lemma \ref{parametrix}). We also prove that $\mathcal{D}_E^{\{M_p\}}=\tilde{\mathcal{D}}_E^{\{M_p\}}$, topologically.

As an application of our theory, we extend the theory of $\mathcal{D}'^{\ast}_{L_{\eta}^{p}}$, $\mathcal{B}'^*_{\eta}$, and $\dot{\mathcal{B}}'^*_{\eta}$ spaces not only by considering quasianalytic cases of $\ast$ but also by allowing ultrapolynomially bounded weights $\eta$ which may growth exponentially. We establish relations among them and make a detailed investigation of their topological properties. We would like to point out that applications of such results to the study of the general convolvability in the setting of quasianalytic ultradistributions will appear elsewhere \cite{ppv}. We conclude this section with some results about convolution and multiplicative products on $\mathcal{D}'^*_{E'_{\ast}}$.
\subsection{Notation}
\label{notation}
Let $(M_p)_{p\in\mathbb{N}}$ and $(A_p)_{p\in\mathbb{N}}$ be two sequences of positive numbers such that $M_0=M_1=A_{0}=A_{1}=1$. Throughout the article, we impose the following assumptions over these weight sequences. The sequence $M_p$ satisfies the ensuing three conditions:\\
\indent $(M.1)$ $M_{p}^{2} \leq M_{p-1} M_{p+1}, \; \; p \in\ZZ_+$;\\
\indent $(M.2)$ $\ds M_{p} \leq c_0H^{p} \min_{0\leq q\leq p} \{M_{p-q} M_{q}\}$, $p,q\in \NN$, for some $c_0,H\geq1$;\\
\indent $(M.5)$ there exists $s>0$ such that $M_p^s$ is strongly non-quasianalytic, i.e., there exists $c_0\geq 1$ such that
\beqs
\sum_{q=p+1}^{\infty}\frac{M_{q-1}^s}{M_q^s}\leq c_0 p\frac{M_p^s}{M_{p+1}^s},\,\, \forall p\in\ZZ_+.
\eeqs
It is clear that if $M_p^s$ is strongly non-quasianalytic than for any $s'> s$, $M_p^{s'}$ is also strongly non-quasianalytic. One easily verifies that when $M_p$ satisfies $(M.5)$ there exists $\kappa>0$ such that $p!^{\kappa}\subset M_p$, i.e. there exist $c_0,L_0>0$ such that $p!^{\kappa}\leq c_0 L_0^p M_p$, $p\in\NN$ (cf. \cite[Lemma 4.1]{Komatsu1}). Following Komatsu \cite{Komatsu1}, for $p\in\ZZ_+$, we denote $m_p=M_p/M_{p-1}$ and for $\rho\geq 0$ let $m(\rho)$ be the number of $m_p\leq \rho$. As a consequence of \cite[Proposition 4.4]{Komatsu1}, by a change of variables, one verifies that $M_p$ satisfies $(M.5)$ if and only if
\beqs
\int_{\rho}^{\infty}\frac{m(\lambda)}{\lambda^{s+1}}d\lambda \leq c \frac{m(\rho)}{\rho^s},\,\, \forall \rho\geq m_1.
\eeqs
A sufficient condition for $M_p$ to satisfies $(M.5)$ is if the sequence $m_p/p^{\lambda}$, $p\in\ZZ_+$, to be monotonically increasing for some $\lambda>0$.

We assume that $A_{p}$ satisfies $(M.1)$ and $(M.2)$. Of course, without losing generality, we can assume that the constants $c_0$ and $H$ from the condition $(M.2)$ are the same for $M_p$ and $A_p$. Moreover, we also assume that $A_p$ satisfies the following additional hypothesis:\\
\indent $(M.6)$ $p!\subset A_p$; i.e., there exist $c_0,L_0>0$ such that $p!\leq c_0 L_0^p A_p$, $p\in\NN$.

 Of course, the constants $c_0$ and $L_0$ in $(M.6)$ can be chosen such that $c_0,L_0\geq 1$. Although it is not part of our assumptions, we will be primary interested in the quasianalytic case, i.e., $\ds \sum_{p=1}^{\infty}\frac{M_{p-1}}{M_p}=\infty$.

We denote by $M(\cdot)$ and $A(\cdot)$ the associated functions of $M_p$ and $A_p$, that is, $\ds M(\rho):=\sup_{p\in\NN}\ln_+\frac{\rho^p}{M_p}$ and $\ds A(\rho):=\sup_{p\in\NN}\ln_+\frac{\rho^p}{A_p}$ for $\rho>0$, respectively. They are non-negative continuous increasing functions (cf. \cite{Komatsu1}). We denote by $\mathfrak{R}$ the set of all positive monotonically increasing sequences which tend to infinity. For $(l_p)\in\mathfrak{R}$, denote by $N_{l_p}$ and $B_{l_p}(\cdot)$ the associated functions of the sequences $M_p\prod_{j=1}^pl_j$ and $A_p\prod_{j=1}^pl_j$, respectively.

 For $h>0$ we denote by $\SSS^{M_p,h}_{A_p,h}$ the Banach spaces (in short $(B)$-space from now on) of all $\varphi \in C^{\infty}(\RR^d)$ for which the norm
\beqs
\sigma_h(\varphi)=\sup_{\alpha} \frac{h^{|\alpha|}\left\|e^{A(h|\cdot|)}D^{\alpha}\varphi\right\|_{L^{\infty}(\RR^d)}}{M_{\alpha}}
\eeqs
is finite. One easily verifies that for $h_1<h_2$ the canonical inclusion $\SSS^{M_p,h_2}_{A_p,h_2}\rightarrow\SSS^{M_p,h_1}_{A_p,h_1}$ is compact. As l.c.s., we define $\ds\SSS^{(M_p)}_{(A_p)}(\RR^{d})=\lim_{\substack{\longleftarrow\\ h\rightarrow \infty}} \SSS^{M_p,h}_{A_p,h}$ and $\ds\SSS^{\{M_p\}}_{\{A_p\}}(\RR^{d})=\lim_{\substack{\longrightarrow\\ h\rightarrow 0}} \SSS^{M_p,h}_{A_p,h}$. Since for $h_1<h_2$ the inclusion $\SSS^{M_p,h_2}_{A_p,h_2}\rightarrow\SSS^{M_p,h_1}_{A_p,h_1}$ is compact, $\SSS^{(M_p)}_{(A_p)}(\RR^{d})$ is an $(FS)$-space and $\SSS^{\{M_p\}}_{\{A_p\}}(\RR^{d})$ is a $(DFS)$-space. In particular they are both Montel spaces. 

For each $(r_p)\in\mathfrak{R}$, by $\SSS^{M_p,(r_p)}_{A_p,(r_p)}$ we denote the space of all $\varphi\in C^{\infty}(\RR^d)$ such that
\beqs
\sigma_{(r_p)}(\varphi)=\sup_{\alpha} \frac{\left\|e^{B_{r_p}(|\cdot|)}D^{\alpha}\varphi\right\|_{L^{\infty}(\RR^d)}}{M_{\alpha}\prod_{j=1}^{|\alpha|}r_j}<\infty.
\eeqs
Provided with the norm $\sigma_{(r_p)}$, the space $\SSS^{M_p,(r_p)}_{A_p,(r_p)}$ becomes a $(B)$-space. Similarly as in \cite{PilipovicK,Pilipovic}, one can prove that $\SSS^{\{M_p\}}_{\{A_p\}}(\RR^d)$ is topologically isomorphic to $\ds \lim_{\substack{(r_p)\in\mathfrak{R}\\ \longleftarrow}}\SSS^{M_p,(r_p)}_{A_p,(r_p)}$.

In the future we shall employ $\SSS^*_{\dagger}(\RR^d)$ as a common notation for $\SSS^{(M_p)}_{(A_p)}(\RR^d)$ (Beurling case) and $\SSS^{\{M_p\}}_{\{A_p\}}(\RR^d)$ (Roumieu case). It is clear that for each $h>0$ and $(r_p)\in\mathfrak{R}$, the spaces $\SSS^{M_p,h}_{A_p,h}$ and $\SSS^{M_p,(r_p)}_{A_p,(r_p)}$ are continuously injected into $\SSS(\RR^d)$ (the Schwartz space).

 We will often make use of the following technical result from \cite{BojanL}.

\begin{lemma}[\cite{BojanL}]\label{nwseq}
Let $(k_p)\in\mathfrak{R}$. There exists $(k'_p)\in\mathfrak{R}$ such that $k'_p\leq k_p$ and $\ds\prod_{j=1}^{p+q}k'_j\leq 2^{p+q}\prod_{j=1}^{p}k'_j\cdot\prod_{j=1}^{q}k'_j$, for all $p,q\in\ZZ_+$.
\end{lemma}

We adopt the following notations. The symbol $``\hookrightarrow"$ stands for a continuous and dense inclusion between topological vector spaces. For $h\in\RR^d$ and $f\in\SSS'^*_{\dagger}(\RR^d)$ we denote as $T_h f $ translation by $h$, i.e., $T_hf=f(\:\cdot\:+h)$. We write $\langle x \rangle=(1+|x|^{2})^{1/2}$, $x\in\mathbb{R}^{d}$.

\section{Some important auxiliary results on the space $\SSS^*_{\dagger}(\RR^d)$}\label{PPV}

We collect in this section some important results on the nuclearity of $\SSS^*_{\dagger}(\RR^d)$, the existence of parametrices as well as a characterisation of bounded sets in $\SSS'^*_{\dagger}(\RR^d)$. These are essential tools in the rest of the article. We refer to \cite{ppv} for the proofs. Unless explicitly stated, we deal with the Beurling and Roumieu cases simultaneously. We follow the ensuing convention. We shall first state assertions for the $(M_p)-(A_p)$ case followed in parenthesis by the corresponding statements for the $\{M_p\}-\{A_p\}$ case.

\begin{proposition}
The space $\SSS^*_{\dagger}(\RR^d)$ is nuclear.
\end{proposition}

\begin{proposition}\label{parametrix}
For every $t>0$ there exist $G\in\SSS^{M_p,t}_{A_p,t}$ and an ultradifferential operator $P(D)$ of class $(M_p)$ (for every $(t_p)\in\mathfrak{R}$ there exist $G\in\SSS^{M_p,(t_p)}_{A_p,(t_p)}$ and an ultradifferential operator $P(D)$ of class $\{M_p\}$) such that $P(D)G=\delta$.
\end{proposition}

\begin{lemma}\label{appincl}
Let $r>0$ ($(r_p)\in\mathfrak{R}$).
\begin{itemize}
\item[$i)$] For each $\chi,\varphi\in \mathcal{S}^*_{\dagger}(\mathbb{R}^d)$ and $\psi\in\SSS^{M_p,r}_{A_p,r}$ ($\psi\in\SSS^{M_p,(r_p)}_{A_p,(r_p)}$), one has  $\chi*(\varphi\psi)\in \SSS^*_{\dagger}(\RR^d)$.
\item[$ii)$] Let $\varphi,\chi\in \mathcal{S}^*_{\dagger}(\mathbb{R}^d)$ with $\varphi(0)=1$ and $\int_{\RR^d}\chi(x)dx=1$. For each $n\in\ZZ_+$ define $\chi_n(x)=n^d\chi(nx)$ and $\varphi_n(x)=\varphi(x/n)$. Then there exists $k\geq 2r$ ($(k_p)\in\mathfrak{R}$ with $(k_p)\leq (r_p/2)$) such that the operators $\tilde{Q}_n:\psi\mapsto\chi_n*(\varphi_n\psi)$, are continuous as mappings from $\SSS^{M_p,k}_{A_p,k}$ into $\SSS^{M_p,r}_{A_p,r}$ (from $\SSS^{M_p,(k_p)}_{A_p,(k_p)}$ into $\SSS^{M_p,(r_p)}_{A_p,(r_p)}$), for all $n\in\ZZ_+$. Moreover $\tilde{Q}_n\rightarrow \mathrm{Id}$ in $\mathcal{L}_b\left(\SSS^{M_p,k}_{A_p,k},\SSS^{M_p,r}_{A_p,r}\right)$ ($\mathcal{L}_b\left(\SSS^{M_p,(k_p)}_{A_p,(k_p)},\SSS^{M_p,(r_p)}_{A_p,(r_p)}\right)$).
\end{itemize}
\end{lemma}

In the next proposition, given $t>0$ ($(t_p)\in\mathfrak{R}$), we denote as $\overline{\SSS}^{M_p,t}_{A_p,t}$ (as $\overline{\SSS}^{M_p,(t_p)}_{A_p,(t_p)}$) the closure of $\SSS^{(M_p)}_{(A_p)}(\RR^d)$ in $\SSS^{M_p,t}_{A_p,t}$ (the closure of $\SSS^{\{M_p\}}_{\{A_p\}}(\RR^d)$ in $\SSS^{M_p,(t_p)}_{A_p,(t_p)}$).

\begin{proposition}\label{parametrix1}
Let $B$ be a bounded subset of $\SSS'^*_{\dagger}(\RR^d)$. There exists $k>0$ ($(k_p)\in\mathfrak{R}$) such that each $f\in B$ can be extended to a continuous functional $\tilde{f}$ on $\overline{\SSS}^{M_p,k}_{A_p,k}$ (on $\overline{\SSS}^{M_p,(k_p)}_{A_p,(k_p)}$). Moreover, there exists $l\geq k$ ($(l_p)\in\mathfrak{R}$ with $(l_p)\leq (k_p)$) such that $\SSS^{M_p,l}_{A_p,l}\subseteq\overline{\SSS}^{M_p,k}_{A_p,k}$ ($\SSS^{M_p,(l_p)}_{A_p,(l_p)}\subseteq \overline{\SSS}^{M_p,(k_p)}_{A_p,(k_p)}$) and $*:\SSS^{M_p,l}_{A_p,l}\times\SSS^{M_p,l}_{A_p,l}\rightarrow \overline{\SSS}^{M_p,k}_{A_p,k}$ ($*:\SSS^{M_p,(l_p)}_{A_p,(l_p)}\times\SSS^{M_p,(l_p)}_{A_p,(l_p)}\rightarrow \overline{\SSS}^{M_p,(k_p)}_{A_p,(k_p)}$) is a continuous bilinear mapping. Furthermore, there exist an ultradifferential operator $P(D)$ of class $*$ and $u\in \overline{\SSS}^{M_p,l}_{A_p,l}$ ($u\in\overline{\SSS}^{M_p,(l_p)}_{A_p,(l_p)}$) such that $P(D)u=\delta$ and $f=(P(D)u)*f=P(D)(u*\tilde{f})$ for each $f\in B$, where $u*\tilde{f}$ is the image of $\tilde{f}$ under the transpose of the continuous mapping $\varphi\mapsto\check{u}*\varphi$, $\SSS^{(M_p)}_{(A_p)}(\RR^d)\rightarrow \overline{\SSS}^{M_p,k}_{A_p,k}$ ($\SSS^{\{M_p\}}_{\{A_p\}}(\RR^d)\rightarrow \overline{\SSS}^{M_p,(k_p)}_{A_p,(k_p)}$). For $f\in B$, $u*\tilde{f}\in L^{\infty}_{e^{A(l|\cdot|)}}\cap C(\RR^d)$ ($u*\tilde{f}\in L^{\infty}_{e^{B_{l_p}(|\cdot|)}}\cap C(\RR^d)$) and in fact $u*\tilde{f}(x)=\langle \tilde{f}, u(x-\cdot)\rangle$. The set $\{u*\tilde{f}|\, f\in B\}$ is bounded in $L^{\infty}_{e^{A(l|\cdot|)}}$ (in $L^{\infty}_{e^{B_{l_p}(|\cdot|)}}$).
\end{proposition}

\begin{lemma}\label{boundedsetS}
Let $B\subseteq\SSS'^*_{\dagger}(\RR^d)$. The following statements are equivalent:
\begin{itemize}
\item[$i)$] $B$ is bounded in $\SSS'^*_{\dagger}(\RR^d)$;
\item[$ii)$] for each $\varphi\in\SSS^*_{\dagger}(\RR^d)$, $\{f*\varphi|\, f\in B\}$ is bounded in $\SSS'^*_{\dagger}(\RR^d)$;
\item[$iii)$] for each $\varphi\in\SSS^*_{\dagger}(\RR^d)$ there exist $t,C>0$ (there exist $(t_p)\in\mathfrak{R}$ and $C>0$) such that $|(f*\varphi)(x)|\leq C e^{A(t|x|)}$ ($|(f*\varphi)(x)|\leq C e^{B_{t_p}(|x|)}$) for all $x\in \RR^d$, $f\in B$;
\item[$iv)$] there exist $C,t>0$ (there exist $(t_p)\in\mathfrak{R}$ and $C>0$) such that
    \beqs
    \left|(f*\varphi)(x)\right|\leq Ce^{A(t|x|)}\sigma_t(\varphi)\,\, \left(\mbox{resp.}\,\,\left|f*\varphi(x)\right|\leq Ce^{B_{t_p}(|x|)}\sigma_{(t_p)}(\varphi)\right)
    \eeqs
    for all $\varphi\in\SSS^*_{\dagger}(\RR^d)$, $x\in\RR^d$, $f\in B$.
\end{itemize}
\end{lemma}

\begin{lemma}\label{cha_s}
Let $f\in \SSS'^{(M_p)}_{(p!)}(\RR^d)$ ($f\in\SSS'^{\{M_p\}}_{\{p!\}}(\RR^d)$). Then $f\in\SSS'^*_{\dagger}(\RR^d)$ if and only if there exists $t>0$ (there exists $(t_p)\in\mathfrak{R}$) such that for every $\varphi\in \SSS^{(M_p)}_{(p!)}(\RR^d)$ (for every $\varphi\in\SSS^{\{M_p\}}_{\{p!\}}(\RR^d)$)
\beqs
\sup_{x\in\RR^d}e^{-A(t|x|)}|(f*\varphi)(x)|<\infty\,\, \left(\sup_{x\in\RR^d}e^{-B_{t_p}(|x|)}|(f*\varphi)(x)|<\infty\right).
\eeqs
\end{lemma}

\section{Translation-invariant Banach spaces of quasianalytic ultradistributions}\label{TIBU}
We extend here the theory of translation-invariant Banach spaces of ultradistributions to the quasianalytic case. We closely follow the approach from \cite{DPV,DPPV}, where the distribution and non-quasianalytic ultradistribution cases were treated.  We mention that some of the arguments below are similar to those from \cite{DPPV}, but for the reader's convenience we include all details about the adaptations in the corresponding proofs.

Let $E$ be a $(B)$-space. We call $E$ a \emph{translation-invariant $(B)$-space of ultradistributions of class $*-\dagger$} if it satisfies the following three axioms:
\begin{itemize}
    \item[(I)] $\mathcal{S}^*_{\dagger}(\mathbb{R}^d)\hookrightarrow E\hookrightarrow \mathcal{S}'^*_{\dagger}(\RR^d)$.
    \item[(II)] $T_{h}(E)\subseteq E$ for each $h\in\RR^d$.
    \item[(III)] There exist $\tau,C>0$ (for every $\tau>0$ there exists $C>0$), such that
    \begin{equation*}
    \|T_hg\|_E\leq C \|g\|_E e^{A(\tau|h|)},\,\,\, \forall h\in\RR^d,\,\, \forall g\in E.
    \end{equation*}
\end{itemize}

Notice that the condition (III) implicitly makes use of the continuity of $T_h$. The next lemma shows that such a continuity is always ensured by the conditions (I) and (II).

\begin{lemma}\label{conoftr}
Let $E$ be a $(B)$-space satisfying (I) and (II). The translation operators $T_h:E\to E$ are bounded for all $h\in\RR^d$.
\end{lemma}

\begin{proof} Observe that $T_h$ is continuous as a mapping from $E$ to $\SSS'^*_{\dagger}(\RR^d)$ since it can be decomposed as $\ds E\xrightarrow{\mathrm{Id}}\SSS'^*_{\dagger}(\RR^d)\xrightarrow{T_h}\SSS'^*_{\dagger}(\RR^d)$ and $T_h:\SSS'^*_{\dagger}(\RR^d)\rightarrow\SSS'^*_{\dagger}(\RR^d)$ is continuous. Thus the graph of $T_h$ is closed in $E\times \SSS'^*_{\dagger}(\RR^d)$ and since its image is in $E$ its graph is also closed in $E\times E$ ($E\times E$ is continuously injected into $E\times \SSS'^*_{\dagger}(\RR^d)$ via the mapping $\mathrm{Id}\times \mathrm{Id}$). As $E$ is a $(B)$-space, the closed graph theorem implies that $T_h$ is continuous.
\end{proof}

\begin{lemma}\label{1111177}
Let $E$ be a translation-invariant $(B)$-space of ultradistributions of class $*-\dagger$. For every $g\in E$, $\ds\lim_{h\to0}\|T_{h}g-g\|_{E}=0$. In particular for each $g\in E$ the mapping $h\mapsto T_h g$, $\RR^d \rightarrow E$, is continuous at $0$ (hence everywhere continuous).
\end{lemma}

\begin{proof} The proof is straightforward and we omit it.
\end{proof}

Summarizing, Lemma \ref{conoftr} and Lemma \ref{1111177} prove that a translation-invariant (B)-space of ultradistributions $E$ of class $*-\dagger$ satisfies the following stronger condition than (II):
\begin{itemize}
\item[$(\widetilde{II})$] for each $h>0$, $T_h:E\rightarrow E$ is continuous and for each $g\in E$ the mapping $h\mapsto T_h g$, $\RR^d\rightarrow E$, is continuous.
\end{itemize}

Clearly $T_0=\mathrm{Id}_E$, $T_{h_1+h_2}=T_{h_1}\circ T_{h_2}=T_{h_2}\circ T_{h_1}$. Next, we define \emph{the weight function $\omega(h)$ of $E$} as
\begin{equation}
\label{ravenstvo1}
\omega(h)={\|T_{-h}\|_{\mathcal{L}(E)}}.
\end{equation}
Obviously the weight function is positive and $\omega(0)=1$. Furthermore, since $\SSS^*_{\dagger}(\RR^d)$ is separable (it is an $(FS)$-space or a $(DFS)$-space, respectively), so is $E$. Thus $\omega(h)=\|T_{-h}\|_{\mathcal{L}(E)}$ is the supremum of $\|T_{-h}g\|_E$ where $g$ belongs to a countable dense subset of the closed unit ball of $E$. Since $h\mapsto \|T_{-h}g\|_E$ is continuous, $\omega$ is measurable. Clearly, the logarithm of $\omega$ is subadditive and there exist $C,\tau>0$ (for every $\tau>0$ there exists $C>0$) such that $\omega(h)\leq C e^{A(\tau|h|)}$.

\begin{remark}
In the Beurling case when $A_p=p!$, the assumption (III) is superfluous. In fact, assuming only (I) and (II), Lemma \ref{conoftr} implies that for each $h\in\RR^d$, $T_h:E\rightarrow E$ is continuous. Additionally, one easily verifies that for each fixed $\varphi\in\SSS^*_{\dagger}(\RR^d)$, one has $T_h\varphi\rightarrow \varphi$ as $h\rightarrow 0$ in $\SSS^*_{\dagger}(\RR^d)$ and consequently in $E$. Hence, employing the same reasoning as above, we obtain that $\omega$ is a measurable positive function with subadditive logarithm. Therefore, there exist $C,h>0$ such that $\omega(h)\leq C e^{k|h|}$, $\forall h\in\RR^d$ (cf. \cite[Sect 7.4]{hilleph}), which is in fact condition (III) in this case.
\end{remark}

We will also give an alternative version of (III) in the Roumieu case which sometimes is easier to work with than (III). For this purpose we need the following technical result from \cite{BojanL}.

\begin{lemma}[\cite{BojanL}]\label{ppp}
Let $g:[0,\infty)\rightarrow[0,\infty)$ be an increasing function that satisfies the following estimate:\\
\indent For every $L>0$ there exists $C>0$ such that $g(\rho)\leq A(L\rho)+ \ln C$.\\
Then there exists a subordinate function $\epsilon(\rho)$ such that $g(\rho)\leq A(\epsilon(\rho))+ \ln C'$, for some constant $C'>1$.
\end{lemma}

See \cite{Komatsu1} for the definition of subordinate function.

\begin{lemma}\label{pppp}
In the Roumieu case condition (III) is equivalent to the following one:
\begin{itemize}
\item[$(\widetilde{III})$] there exist $(l_p)\in\mathfrak{R}$ and $C>0$ such that $\|T_h g\|_E\leq C \|g\|_E e^{B_{l_p}(|h|)}$, for all $g\in E$, $h\in\RR^d$.
\end{itemize}
\end{lemma}
\begin{proof}
The proof is analogous to that of $(c)\Leftrightarrow(\tilde{c})$ in \cite[Theorem 4.2]{DPPV}.
\end{proof}

The next theorem gives a weak criterion to conclude that a $(B)$-space $E$ is a translation-invariant space of ultradistributions of class $*-\dagger$.

\begin{theorem}
Let $E$ be a $(B)$-space satisfying:
\begin{itemize}
\item[(I)$'$] $\mathcal{S}^{(M_p)}_{(p!)}(\mathbb{R}^d)\hookrightarrow E\hookrightarrow \mathcal{S}'^{(M_p)}_{(p!)}(\RR^d)$ ($\mathcal{S}^{\{M_p\}}_{\{p!\}}(\mathbb{R}^d)\hookrightarrow E\hookrightarrow \mathcal{S}'^{\{M_p\}}_{\{p!\}}(\RR^d)$);
\item[(II)] $T_h(E)\subseteq E$, for all $h\in\mathbb{R}^{n}$;
\item[(III)$'$] for any $g\in E$ there exist $C=C_g>0$ and $\tau=\tau_g>0$ (for every $\tau>0$ there exists $C=C_{g,\tau}>0$) such that $\|T_hg\|_E\leq C e^{A(\tau|h|)}$, $\forall h\in\RR^d$.
\end{itemize}
Then $E$ is a translation-invariant $(B)$-space of ultradistributions of class $*-\dagger$.
\end{theorem}

\begin{proof} Employing the same technique as in the proof of Lemma \ref{conoftr}, one easily verifies that conditions (I)$'$ and (II) imply the continuity of $T_h:E\rightarrow E$. The proof of (III) can be obtain by adapting the proof of $(c)$ in \cite[Theorem 4.2]{DPPV}.

 We now address $(I)$. To prove $\SSS^*_{\dagger}(\RR^d)\hookrightarrow E$, by $(I)'$, it is enough to prove that $\SSS^{*}_{\dagger}(\RR^d)$ is continuously injected into $E$. Pick $\psi_1\in\DD(\RR^d)$ such that
\beqs
\sum_{m\in\ZZ^d}\psi_1(x-m)=1,\,\, \forall x\in\RR^d,\,\, \mathrm{supp}\:\psi_1\in[-1,1]^d
\eeqs
and $\psi_1$ is non-negative and even. Next, pick $\psi_2\in \SSS^{(M_p)}_{(p!)}(\RR^d)$ ($\psi_2\in\SSS^{\{M_p\}}_{\{p!\}}(\RR^d)$), such that $\int_{\RR^d}\psi_2(x)dx=1$ and $\psi_2$ is even. Set $\psi=\psi_1*\psi_2$. One readily verifies that $\sum_{m\in\ZZ^d}\psi(x-m)=1$ for all $x\in\RR^d$ and $\psi\in \SSS^{(M_p)}_{(p!)}(\RR^d)$ in the Beurling case and $\psi\in \SSS^{\{M_p\}}_{\{p!\}}(\RR^d)$ in the Roumieu case, respectively. By (III), there exist $C,\tau>0$ (for every $\tau>0$ there exists $C>0$) such that
\beq\label{boundsfor1}
\|\varphi T_{-m}\psi\|_E\leq C e^{-A(\tau|m|)}\|e^{2A(\tau|m|)}\psi T_m \varphi\|_E,\,\, \forall \varphi\in\SSS^*_{\dagger},\, \forall m\in\ZZ^d.
\eeq
For $m\in\ZZ^d$, consider the linear mapping $\rho_{m,\tau}(\varphi)=e^{2A(\tau|m|)}\psi T_m \varphi$, $\SSS^{(M_p)}_{(A_p)}(\RR^d)\rightarrow \SSS^{(M_p)}_{(p!)}(\RR^d)$ ($\SSS^{\{M_p\}}_{\{A_p\}}(\RR^d)\rightarrow \SSS^{\{M_p\}}_{\{p!\}}(\RR^d)$). Clearly, it is well defined. Let $B$ be a bounded subset of $\SSS^*_{\dagger}(\RR^d)$. Then for every $h>0$ (there exists $h>0$) such that
\beq\label{bb119}
\sup_{\varphi\in B}\sup_{\alpha\in\NN^d}\frac{h^{|\alpha|}\left\|e^{A(h|\cdot|)}D^{\alpha}\varphi\right\|_{L^{\infty}}}{M_{\alpha}}<\infty
\eeq
Now, \cite[Lemma 3.6]{Komatsu1} implies
\beq\label{rrss7}
e^{2A(\tau|m|)}\leq 2c_0 e^{A(2H\tau|x+m|)}e^{A(2H\tau|x|)}.
\eeq
In the Beurling case, let $h_1>0$ be arbitrary but fixed. Choose $h>0$ such that $h\geq \max\{2H\tau,2h_1\}$ and $e^{A(2H\tau\lambda)}\leq C' e^{h\lambda}$ for all $\lambda\geq 0$ (such an $h$ exists because $p!\subset A_p$). By (\ref{bb119}) and (\ref{rrss7}) we have
\beq\label{vvkk3}
\frac{h_1^{|\alpha|}\left|D^{\alpha}\left(\psi(x)T_m\varphi(x)\right)\right|e^{h_1|x|}}{M_{\alpha}}\leq C_2e^{-2A(\tau|m|)},
\eeq
for all $x\in\RR^d$, $m\in\ZZ^d$, $\varphi\in B$. Hence $\{\rho_{m,\tau}|\, m\in\ZZ^d\}$ is uniformly bounded on $B$. In the Roumieu case there exist $\tilde{h},\tilde{C}>0$ such that $\tilde{h}^{|\alpha|}\left|D^{\alpha}\psi(x)\right|e^{\tilde{h}|x|}\leq \tilde{C}M_{\alpha}$ for all $x\in\RR^d$, $\alpha\in\NN^d$. For the $h>0$ for which (\ref{bb119}) holds choose $0<\tau \leq h/(2H)$ such that $e^{A(2H\tau\lambda)}\leq C' e^{\tilde{h}\lambda/2}$ for all $\lambda\geq 0$ (such a $\tau$ exists because $p!\subset A_p$). Choose $h_1\leq \min\{h/2,\tilde{h}/2\}$. Then, by using (\ref{bb119}) and (\ref{rrss7}), similarly as in the Beurling case, we obtain (\ref{vvkk3}), i.e., $\{\rho_{m,\tau}|\, m\in\ZZ^d\}$ is uniformly bounded on $B$. Now, (I)$'$ implies that $\|\rho_{m,\tau}(\varphi)\|_E\leq C'_2$ for all $\varphi\in B$, $m\in\ZZ^d$. By using (\ref{boundsfor1}), we obtain that the sequence $\left\{\sum_{|m|\leq N} \varphi T_{-m}\psi\right\}_{N=0}^{\infty}$ is a Cauchy sequence in $E$ for each $\varphi\in B$. Since its limit is $\varphi$ in $\SSS'^{(M_p)}_{(p!)}(\RR^d)$ (in $\SSS'^{\{M_p\}}_{\{p!\}}(\RR^d)$) it converges to $\varphi\in E$. Also $\|\varphi\|_E\leq C$ for all $\varphi\in B$. This implies that $\SSS^*_{\dagger}(\RR^d)\subseteq E$ and the inclusion maps bounded sets into bounded sets. As $\SSS^*_{\dagger}(\RR^d)$ is bornological, the inclusion is continuous. It remains to prove $E\subseteq \SSS'^*_{\dagger}(\RR^d)$. By (I)$'$ for a bounded set $B$ in $\SSS^{(M_p)}_{(p!)}(\RR^d)$ (in $\SSS^{\{M_p\}}_{\{p!\}}(\RR^d)$) there exists $D>0$ such that $|\langle g,\check{\varphi}\rangle|\leq D\|g\|_E$ for all $g\in E$ and $\varphi\in B$. Then (III) implies that there exist $C,\tau>0$ (for every $\tau>0$ there exists $C>0$) such that
\beqs
\left|(g*\varphi)(y)\right|\leq D\|T_y g\|_E\leq CD e^{A(\tau|y|)},\,\, \mbox{for all}\,\, y\in\RR^d,\, \varphi\in B,\, g\in E.
\eeqs
In the Beurling case Lemma \ref{cha_s} implies $E\subseteq \SSS'^{(M_p)}_{(A_p)}(\RR^d)$. In the Roumieu case Lemma \ref{cha_s} together with Lemma \ref{pppp} implies $E\subseteq \SSS'^{\{M_p\}}_{\{A_p\}}(\RR^d)$. Since $E\rightarrow \SSS'^{(M_p)}_{(p!)}(\RR^d)$ is continuous ($E\rightarrow \SSS'^{\{M_p\}}_{\{p!\}}(\RR^d)$ is continuous) it has a closed graph. Thus the inclusion $E\rightarrow \SSS'^*_{\dagger}(\RR^d)$ has a closed graph. As $\SSS'^{(M_p)}_{(A_p)}(\RR^d)$ is a $(DFS)$-space ($\SSS'^{\{M_p\}}_{\{A_p\}}(\RR^d)$ is an $(FS)$-space), it is a Pt\'{a}k space (cf. \cite[Sect. IV. 8, p. 162]{Sch}). Thus the continuity of $E\rightarrow \SSS'^*_{\dagger}(\RR^d)$ follows from the Pt\'{a}k closed graph theorem (cf. \cite[Thm. 8.5, p. 166]{Sch}).
\end{proof}

Throughout the rest of the article we shall \emph{always assume} that $E$ is a translation-invariant $(B)$-spaces of ultradistributions of class $*-\dagger$. Our next concern is the study of convolution structures on $E$. We need three technical lemmas.

\begin{lemma}\label{bocintl}
Let $\varphi\in\SSS^*_{\dagger}(\RR^{2d})$. Then for each $y\in \RR^d$, $\varphi(\cdot,y)\in \SSS^*_{\dagger}(\RR^d)$ and the function $\ds\psi(x)=\int_{\RR^d}\varphi(x,y)dy$ is an element of $\SSS^*_{\dagger}(\RR^d)$. Moreover, the function $\mathbf{f}:\RR^d\rightarrow E$, $y\mapsto \varphi(\cdot,y)$, is Bochner integrable and $\ds\psi=\int_{\RR^d}\mathbf{f}(y)dy$.
\end{lemma}

\begin{proof} The fact that $\varphi(\cdot,y)\in \SSS^*_{\dagger}(\RR^d)$ for each $y\in\RR^d$ and that $\psi\in\SSS^*_{\dagger}(\RR^d)$ is trivial. Thus $\mathbf{f}$ is well defined on $\RR^d$ with values in $E$ (in fact its values are in $\SSS^*_{\dagger}(\RR^d)$). One easily verifies that $\mathbf{f}$ is continuous, hence strongly measurable. To prove that it is Bochner integrable it remains to prove that $y\mapsto \|\mathbf{f}(y)\|_E$ is in $L^1(\RR^d)$. The condition (I) implies
\beqs
\|\mathbf{f}(y)\|_E\leq C_1\sup_{\alpha}\frac{m^{|\alpha|} \left\|e^{A(m|\cdot|)}D^{\alpha}_x\varphi(\cdot,y)\right\|_{L^{\infty}(\RR^d_x)}}{M_{\alpha}}\leq C_2\sigma_{mH}(\varphi)e^{-A(m|y|)}.
\eeqs
Thus $\mathbf{f}$ is Bochner integrable. Now, for $n\in\ZZ_+$, denote $K_n=[-n,n]^d$. Since $K_n$ is compact and $\mathbf{f}$ is continuous there exists $l(n)\in\ZZ_+$ such that $l(n)\geq n$ and $\|\mathbf{f}(y)-\mathbf{f}(y')\|_E\leq 2^{-n}$ when $y,y'\in K_n$ and $|y_j-y'_j|\leq 1/l(n)$, $j=1,...,d$. Of course we can take $l(n+1)> l(n)$, for all $n\in\ZZ_+$. Set $D_n=\{y\in K_n|\, y=(k_1/l(n),...,k_d/l(n)), k_j\in \ZZ, -nl(n)\leq k_j\leq nl(n)-1, j=1,...,d\}$ and let
\beqs
L_n(x)=\sum_{t\in D_n}\varphi(x,t)l(n)^{-d}.
\eeqs
Clearly $L_n\in\SSS^*_{\dagger}(\RR^d)\subseteq E$. We prove that $L_n\rightarrow \psi$ when $n\rightarrow\infty$, in $\SSS^*_{\dagger}(\RR^d)$. We give the proof for the Roumieu case, the Beurling case being similar. There exists $m>0$ such that $\varphi\in \SSS^{M_p,m}_{A_p,m}(\RR^d)$. Pick $m'>0$ such that $m'\leq m/(2H^2)$. For each $t=(t_1,...,t_d)\in D_n$ denote $K_{n,t}=[t_1,t_1+1/l(n))\times...\times[t_d,t_d+1/l(n))$. Observe that\\
\\
$\left|D^{\alpha}\psi(x)-D^{\alpha}L_n(x)\right|$
\beqs
\leq \int_{\RR^d\backslash K_n}\left|D^{\alpha}_x\varphi(x,y)\right|dy+ \sum_{t\in D_n} \int_{K_{n,t}}\left|D^{\alpha}_x\varphi(x,y)-D^{\alpha}_x\varphi(x,t)\right|dy=S_1(x)+S_2(x).
\eeqs
For $y\in K_{n,t}$, by Taylor expanding $D^{\alpha}_x\varphi(x,y)$ at $(x,t)$, we have\\
\\
$\left|D^{\alpha}_x\varphi(x,y)-D^{\alpha}_x\varphi(x,t)\right|$
\beqs
&\leq& \sum_{|\beta|=1}\left|(y-t)^{\beta}\right|\int_0^1 \left|D^{\alpha}_xD^{\beta}_y\varphi(x,t+s(y-t))\right|ds\\
&\leq& \frac{d\sigma_m(\varphi)}{n}\cdot m^{-|\alpha|-1}M_{|\alpha|+1}\int_0^1 e^{-A(m|(x,t+s(y-t))|)}ds.
\eeqs
By \cite[Proposition 3.6]{Komatsu1} and the fact $e^{A(\rho+\mu)}\leq 2e^{A(2\rho)}e^{A(2\mu)}$, for $\rho,\mu>0$ (which can be easily verified), we have
\beqs
e^{A(m'|x|)}e^{A(m'|y|)}&\leq& 2e^{A(m'|x|)}e^{A(2m'|t+s(y-t)|)}e^{A(2m'|(1-s)(y-t)|)}\\
&\leq& c_1e^{A(2m'H|(x,t+s(y-t))|)}.
\eeqs
Hence
\beqs
\left|D^{\alpha}_x\varphi(x,y)-D^{\alpha}_x\varphi(x,t)\right|\leq \frac{C_1}{n}\cdot \frac{H^{|\alpha|}M_{\alpha}}{m^{|\alpha|}e^{A(m'|x|)}e^{A(m'|y|)}}\leq \frac{C_1M_{\alpha}}{nm'^{|\alpha|}e^{A(m'|x|)}e^{A(m'|y|)}}.
\eeqs
Thus, for $S_2(x)$ we have the following estimate
\beq\label{fors2}
S_2(x)\leq \frac{C_1M_{\alpha}}{nm'^{|\alpha|}e^{A(m'|x|)}}\int_{\RR^d}e^{-A(m'|y|)}dy\leq \frac{C_2M_{\alpha}}{nm'^{|\alpha|}e^{A(m'|x|)}}.
\eeq
To estimate $S_1$, we proceed as follows
\beqs
S_1(x)\leq \frac{\sigma_m(\varphi)M_{\alpha}}{m^{|\alpha|}}\int_{\RR^d\backslash K_n}e^{-A(m|(x,y)|)}dy.
\eeqs
For $y\in \RR^d\backslash K_n$, by \cite[Proposition 3.6]{Komatsu1}, we have
\beqs
e^{A(m'n)}e^{A(m'|x|)}e^{A(m'|y|)}\leq c_0e^{A(m'|x|)}e^{A(m'H|y|)}\leq c_0^2 e^{A(m'H^2|(x,y)|)}.
\eeqs
Hence
\beq
S_1(x)&\leq& \frac{C_3M_{\alpha}}{m^{|\alpha|}e^{A(m'|x|)}e^{A(m'n)}}\int_{\RR^d} e^{-A(m'|y|)}dy\\
&\leq& \frac{C_4M_{\alpha}}{m'^{|\alpha|}e^{A(m'|x|)}e^{A(m'n)}}.\label{fors1}
\eeq
Now, (\ref{fors2}) and (\ref{fors1}) imply that $L_n\rightarrow \psi$ in $\SSS^{M_p,m'}_{A_p,m'}(\RR^d)$ and hence also in $\SSS^{\{M_p\}}_{\{A_p\}}(\RR^d)$. As we noted, the Beurling case is completely analogous. By (I) this also implies $L_n\rightarrow \psi$ in $E$. Denote by $\chi_{n,t}$ the characteristic function of $K_{n,t}$ and define
\beqs
\mathbf{L}_n(y)=\sum_{t\in D_n}\mathbf{f}(t)\chi_{n,t}(y),\,\, y\in\RR^d.
\eeqs
Then $\mathbf{L}_n$ is a simple function on $\RR^d$ with values in $E$ and $\ds\int_{\RR^d}\mathbf{L}_n(y)dy=L_n$. By using the continuity of $\mathbf{f}$ one easily verifies that $\mathbf{L}_n$ converges pointwisely to $\mathbf{f}$. Moreover, by the definition of $K_{n,t}$ we have $\|\mathbf{L}_n(y)\|_E\leq \|\mathbf{f}(y)\|_E +2^{-n}$, for $y\in K_n$ and for $y\not \in K_n$, $\mathbf{L}_n(y)=0$. Thus, by defining $g(y)=1/2$ for $y\in K_1$ and $g(y)=2^{-n}$ when $y\in K_n\backslash K_{n-1}$ for $n\in\ZZ_+$, $n\geq 2$, we obtain $\|\mathbf{L}_n(y)\|_E\leq \|\mathbf{f}(y)\|_E+g(y)$ for all $y\in\RR^d$. Since $g\in L^1(\RR^d)$ and $\mathbf{f}$ is Bochner integrable, dominated convergence implies
\beqs
\lim_{n\rightarrow\infty}L_n=\lim_{n\rightarrow\infty}\int_{\RR^d}\mathbf{L}_n(y)dy=\int_{\RR^d}\mathbf{f}(y)dy,
\eeqs
which completes the proof.
\end{proof}

\begin{lemma}\label{lemma:ime2}
The convolution mapping $(\varphi,\psi)\in \mathcal{S}^*_{\dagger}(\mathbb{R}^d)\times \mathcal{S}^*_{\dagger}(\mathbb{R}^d)\rightarrow
\varphi\ast\psi \in \mathcal{S}^*_{\dagger}(\mathbb{R}^d) $ extends to a continuous bilinear mapping $\mathcal{S}^*_{\dagger}(\mathbb{R}^d)\times E\rightarrow E$. Furthermore, the following estimate holds
\beq\label{1}
\|\varphi\ast g\|_E\leq \|g\|_E\int_{\mathbb{R}^d}|\varphi(x)|\:\omega(x)dx.
\eeq
\end{lemma}

\begin{proof} Let $\varphi,\psi\in\SSS^*_{\dagger}(\mathbb{R}^{d})$. One easily verifies that the function $f(x,y)=\varphi(y)\psi(x-y)$ is an element of $\SSS^*_{\dagger}(\RR^{2d})$. Define $\mathbf{f}:\RR^d\rightarrow E$, $\mathbf{f}(y)=f(\cdot, y)=\varphi(y)T_{-y}\psi$. Then, by Lemma \ref{bocintl}, $\mathbf{f}$ is Bochner integrable and
\beqs
\varphi*\psi=\int_{\RR^d}\mathbf{f}(y)dy.
\eeqs
Observe that $\|\mathbf{f}(y)\|_E\leq |\varphi(y)|\omega(y)\|\psi\|_E$. Thus, we have
\beqs
\|\varphi*\psi\|_E\leq \int_{\RR^d}\|\mathbf{f}(y)\|_Edy\leq \|\psi\|_E\int_{\RR^d}|\varphi(y)|\omega(y)dy,
\eeqs
which proves (\ref{1}) for $g\in\SSS^*_{\dagger}(\RR^d)$. For general $g\in E$, (\ref{1}) follows from a standard density argument. The continuity of the convolution as a bilinear mapping $\mathcal{S}^*_{\dagger}(\mathbb{R}^d)\times E\rightarrow E$ in the Beurling case is an easy consequence of (\ref{1}). In the Roumieu case, from (\ref{1}) we can conclude separate continuity, but $\SSS^{\{M_p\}}_{\{A_p\}}(\RR^d)$ and $E$ are barreled $(DF)$-spaces, hence the separate continuity implies the continuity of the convolution.
\end{proof}

\begin{lemma}\label{densitySL}
$\SSS^*_{\dagger}(\RR^d)$ is dense in $L^1_{\omega}$.
\end{lemma}

\begin{proof} Observe that $C_c(\RR^d)$ (the space of continuous functions with compact support) is dense in $L^1_{\omega}$. Thus it is enough to prove that each $\psi\in C_c(\RR^d)$ can be approximated by elements of $\SSS^*_{\dagger}(\RR^d)$ in $L^1_{\omega}$. Let $\psi\in L^1_{\omega}$. Select a nonzero $\varphi\in\SSS^*_{\dagger}(\RR^d)$ such that $\ds\int_{\RR^d}\varphi(x)dx=1$. For $n\in\ZZ_+$, set $\varphi_n(x)=n^d \varphi(nx)$. One easily verifies that $\varphi_n*\psi\in\SSS^*_{\dagger}(\RR^d)$. We prove $\varphi_n*\psi\rightarrow \psi$ in $L^1_{\omega}(\RR^d)$. We consider the Roumieu case, as the Beurling case is analogous. By $(\widetilde{III})$ there exist $(l_p)\in\mathfrak{R}$ and $C'>0$ such that $\omega(x)\leq C' e^{B_{l_p}(|x|)}$. By Lemma \ref{nwseq} we can assume that $(l_p)$ satisfies $\prod_{j=1}^{p+q}l_j\leq 2^{p+q}\prod_{j=1}^{p}l_j\cdot\prod_{j=1}^{q}l_j$, for all $p,q\in\ZZ_+$. Let $r_p=l_p/4H$, $p\in\ZZ_+$. Since $\varphi\in\SSS^{\{M_p\}}_{\{A_p\}}(\RR^d)$, $|\varphi(x)|\leq C'' e^{-B_{r_p}(|x|)}$. Observe that
\beqs
\omega(x)\left|(\varphi_n*\psi)(x)-\psi(x)\right|&\leq& \omega(x)\int_{\RR^d}|\varphi(y)|\left|\psi(x-y/n)-\psi(x)\right|dy\\
&\leq& Ce^{B_{l_p}(|x|)}\int_{\RR^d}e^{-B_{r_p}(|y|)}\left|\psi(x-y/n)-\psi(x)\right|dy.
\eeqs
Since $\psi$ has compact support $e^{B_{l_p}(|x|)}e^{-B_{r_p}(|y|)}|\psi(x)|\in L^1(\RR^{2d}_{x,y})$ and
\beqs
e^{B_{l_p}(2|x|)}|\psi(x)|\leq C_1 \langle x\rangle^{-d-1},\,\, \forall x\in\RR^d.
\eeqs
This inequality, together with $e^{B_{l_p}(\rho+\mu)}\leq 2e^{B_{l_p}(2\rho)}e^{B_{l_p}(2\mu)}$, $\rho,\mu>0$, implies
\beqs
e^{B_{l_p}(|x|)}|\psi(x-y/n)|&\leq& 2e^{B_{l_p}(2|x-y/n|)}e^{B_{l_p}(2|y/n|)}|\psi(x-y/n)|\\
&\leq& C_1\langle x-y/n\rangle^{-d-1}e^{B_{l_p}(2|y|)}\leq C_2\langle x\rangle^{-d-1}\langle y\rangle^{d+1}e^{B_{l_p}(2|y|)}\\
&\leq& C_3\langle x\rangle^{-d-1}\langle y\rangle^{-d-1}e^{2B_{l_p}(2|y|)}.
\eeqs
Since the sequence $A_p\prod_{j=1}^p l_j$ satisfies $(M.2)$ with the constant $2H$ instead of $H$, \cite[Proposition 3.6]{Komatsu1} implies $e^{2B_{l_p}(2|y|)}\leq c'e^{B_{r_p}(|y|)}$ (by definition of $(r_p)$). Thus
\beqs
e^{B_{l_p}(|x|)}e^{-B_{r_p}(|y|)}\left|\psi(x-y/n)\right|\leq C_4 \langle x\rangle^{-d-1}\langle y\rangle^{-d-1}\in L^1(\RR^{2d}_{x,y}).
\eeqs
Since $\psi$ is continuous, $e^{B_{l_p}(|x|)}e^{-B_{r_p}(|y|)}\left|\psi(x-y/n)-\psi(x)\right|\rightarrow 0$ as $n\rightarrow\infty$ pointwise. Hence, dominated convergence implies $\varphi_n*\psi-\psi\rightarrow0$ as $n\rightarrow\infty$ in $L^1_{\omega}$.
\end{proof}

Combining Lemmas \ref{lemma:ime2} and \ref{densitySL}, we immediately obtain the ensuing important proposition.

\begin{proposition}\label{cor:Beurling Algebra}
The convolution extends as a mapping  $L^{1}_{\omega}\times E\rightarrow E$ and $E$ becomes a Banach module over the Beurling
algebra $L^{1}_{\omega}$, i.e., $ \|u\ast g\|_{E}\leq \|u\|_{1,\omega}\|g\|_{E}$.
\end{proposition}

\begin{corollary}\label{molifaer}
    Let $g\in E$ and $\varphi\in \mathcal{S}^*_{\dagger}(\mathbb{R}^d)$. Set $\ds \varphi_{\varepsilon}(x)=\varepsilon^{-d}\varphi
    \left(x/ \varepsilon \right)$. Then, $$ \lim_{\varepsilon\to 0^+} \|cg-\varphi_\varepsilon*g\|_E=0,$$ where
    $\ds c=\int_{\mathbb{R}^d}\varphi(x)dx$.
\end{corollary}

\begin{proof} Let $0<\varepsilon<1$. We first consider the case when $\varphi,g\in \SSS^*_{\dagger}(\RR^d)$. Observe that
\beqs
cg(x)-(\varphi_\varepsilon*g)(x)=\int_{\RR^d}\left(g(x)-g(x-\varepsilon y)\right)\varphi(y)dy.
\eeqs
One easily verifies that the function $f_{\varepsilon}(x,y)=\left(g(x)-g(x-\varepsilon y)\right)\varphi(y)$ is in $\SSS^*_{\dagger}(\RR^{2d})$. Define $\mathbf{f}_{\varepsilon}(y)=f_{\varepsilon}(\cdot,y)=\left(g-T_{-\varepsilon y}g\right)\varphi(y)$, $\RR^d\rightarrow E$. Lemma \ref{bocintl} implies that $\mathbf{f}_{\varepsilon}$ is Bochner integrable and
\beq\label{inq11}
\|cg-\varphi_{\varepsilon}*g\|_E=\left\|\int_{\RR^d}\mathbf{f}_{\varepsilon}(y)dy\right\|_E\leq\int_{\RR^d} \left\|g-T_{-\varepsilon y}g\right\|_E|\varphi(y)|dy.
\eeq
Clearly $\left\|g-T_{-\varepsilon y}g\right\|_E|\varphi(y)|\leq \|g\|_E\left(1+Ce^{A(m|y|)}\right)|\varphi(y)|$ for some $C,m>0$ (for each $m>0$ and a corresponding $C=C_m$). Since the left hand side is in $L^1(\RR^d)$ and, by Lemma \ref{1111177}, $\left\|g-T_{-\varepsilon y}g\right\|_E|\varphi(y)|\rightarrow 0$ for each fixed $y\in\RR^d$ as $\varepsilon\rightarrow 0^+$, dominated convergence together with (\ref{inq11}) proves the corollary. Due to the density of $\mathcal{S}^*_{\dagger}(\mathbb{R}^{d})\hookrightarrow E$, the conclusion in the lemma for $g\in E$ and $\varphi\in\SSS^*_{\dagger}(\RR^d)$ follows by using the estimate (\ref{1}).
\end{proof}

\begin{proposition}\label{prop3.2}
The space $E'$ satisfies
 \begin{itemize}
        \item [$a')$] $\mathcal{S}^*_{\dagger}(\mathbb{R}^d)\to E'\hookrightarrow \mathcal{S}'^*_{\dagger}(\mathbb{R}^d)$, with continuous embeddings.
        \item [$b')$] For each $h$, $T_h: E'\rightarrow E'$ is a bounded operator. The mappings $\mathbb{R}^d\to E'$, given by $h\mapsto T_hf$, are continuous for the weak$^{\ast}$ topology.
    \end{itemize}
Moreover, the property (III) holds true when $E$ is replaced by $E'$.
\end{proposition}

\begin{proof}
The proof is similar to that of \cite[Proposition 2]{DPV}.
\end{proof}

We can now associate a Beurling algebra to $E'$. Set
$$\check{\omega}(h):=\|T_{-h}\|_{\mathcal{L}(E')}=\|T^{\top}_{h}\|_{\mathcal{L}(E')}=\omega(-h).$$
The very last equality follows from the well-known property $\|T^{\top}_{h}\|_{\mathcal{L}(E')}=\|T_{h}\|_{\mathcal{L}(E)}$, which is of course a consequence of the bipolar theorem (cf. \cite[p. 160]{Sch}). The associated Beurling algebra to the dual space $E'$ is $L^{1}_{\check{\omega}}$. We define the convolution $u\ast f=f\ast u$ of $f\in E'$ and $u\in L^{1}_{\check{\omega}}$ via transposition:
\begin{equation}
\label{eq:convolution}
\left\langle u\ast f,g \right\rangle:= \left\langle f,\check {u}\ast g\right\rangle, \ \ \ g\in E.
\end{equation}
In view of Proposition \ref{cor:Beurling Algebra}, this convolution is well-defined because $\check {u}\in L^{1}_{\omega}$.

\begin{corollary}\label{cor3.3}
We have $\|u\ast f\|_{E'}\leq \|u\|_{1,\check{\omega}}\|f\|_{E'}$ and thus $E'$ is a Banach module
over the Beurling algebra $L^{1}_{\check{\omega}}$. In addition, if $\varphi_{\varepsilon}$ and $c$ are as in Corollary
\ref{molifaer}, then $\varphi_{\varepsilon}\ast f \to c f$ as $\varepsilon \to 0^+$ weakly$^{\ast}$ in $E'$ for each fixed
$f\in E'$.
\end{corollary}

\begin{proof}
For $g\in E$ fixed we have $\langle\varphi_{\varepsilon}*f-cf,g\rangle=\langle f,\check{\varphi_{\varepsilon}}*g-cg\rangle$.
\end{proof}

In general the embedding $\mathcal{S}^*_{\dagger}(\mathbb{R}^{d})\to E'$ is not dense (consider for instance $E= L^{1}$). However, $E'$ inheres the three properties (I), (II), and (III) whenever $E$ is reflexive. The following result is a direct consequence of Proposition \ref{prop3.2} and the Hahn-Banach theorem.

\begin{proposition}\label{prop3.3}
If $E$ is reflexive, then its dual space $E'$ is also a translation-invariant $(B)$-space of ultradistributions of class $*-\dagger$.
\end{proposition}

The fact that the mappings $h\mapsto T_h f$, $\RR^d\rightarrow E'$ do not have to be necessarily continuous in the non-reflexive case ($E=L^{1}(\mathbb{R}^{d})$ is an example) causes various difficulties when dealing with this space. As in the non-quasianalytic case \cite{DPV,DPPV}, we will often work with the closed subspace $E'_{\ast}$ of $E'$ from the following definition rather than with $E'$ itself.

\begin{definition}\label{goodspace}
The $(B)$-space $E'_{\ast}$ stands for $L^{1}_{\check{\omega}}\ast E'$.
\end{definition}

Note that $E'_{\ast}$ is a closed linear subspace of $E'$, due to the Cohen-Hewitt factorization theorem \cite{kisynski} and the fact that $L^{1}_{\check{\omega}}$ possesses bounded approximation unities.

\begin{remark}\label{r11}
Observe that $\SSS^*_{\dagger}(\RR^d)$ is a subset of the closure of $\operatorname*{span} (\SSS^*_{\dagger}(\RR^d) * \SSS^*_{\dagger}(\RR^d))$ in $E'$, where $\mathrm{span}(A)$ denotes the linear span of a set. To see this, let $\varphi\in\SSS^*_{\dagger}(\RR^d)$. Then, if $\chi_n$, $n\in\ZZ_+$, is a $\delta$-sequence from $\SSS^*_{\dagger}(\RR^d)$, $\chi_n*\varphi\rightarrow \varphi$ in $\SSS^*_{\dagger}(\RR^d)$ hence also in $E'$ (by $a')$ of Proposition \ref{prop3.2}). Whence we also obtain that $\SSS^*_{\dagger}(\RR^d)\subseteq E'_{\ast}$.
\end{remark}

The space $E'_{\ast}$ will be of crucial importance throughout the rest of this work. It possesses richer properties than $E'$
with respect to the translation group, as stated in the next theorem.

\begin{theorem}
\label{thgoodspace} The space $E'_{\ast}$ has the properties $a')$, $(\widetilde{II})$ and (III). It is a Banach module over the Beurling algebra $L^{1}_{\check{\omega}}$. If $\varphi_{\varepsilon}$ and $c$ are as in Corollary \ref{molifaer}, then, for each $f\in E'_{\ast}$,
\begin{equation}
\label{eqapprox}
\lim_{\varepsilon\to0^{+}}\|cf-\varphi_{\varepsilon}\ast f\|_{E'}=0.
\end{equation}
Furthermore, if $E$ is reflexive, then $E'_{\ast}=E'$.
\end{theorem}

\begin{proof}
The proof goes in the same lines as that of \cite[Theorem 4.4]{DPPV}
\end{proof}

We point out that (\ref{eqapprox}) implies that $\mathcal{S}^*_{\dagger}(\mathbb{R}^{d})\ast E'\subseteq L^{1}_{\check{\omega}}\ast E'$ is dense in $E'_{\ast}$. In fact, $E'_{\ast}$ is the biggest subspace of $E'$ where the mappings $h\mapsto T_h f$, $\RR^d\rightarrow E'$, are continuous. The proof of this result is essentially the same as that of \cite[Theorem 4.4]{DPPV}, so we omit it.

\begin{proposition}\label{prop3.11}
We have $\ds E'_{\ast}=\left\{f\in E'\bigg|\, \lim_{h\to 0}\|T_{h}f-f\|_{E'}=0\right\}$.
\end{proposition}

In view of property $b')$ from Proposition \ref{prop3.2}, we can naturally define a convolution mapping $E'\times \check{E} \to C(\mathbb{R}^{d})$, where $\check{E}=\left\{g\in \mathcal{S}'^*_{\dagger}(\mathbb{R}^{d})\big|\, \check{g}\in E\right\}$ with norm $\|g\|_{\check{E}}:=\|\check{g}\|_{E}$, via
$$(f*g)(x)=\langle f(t),g(x-t)\rangle=\langle f(t),T_{-x}\check{g}(t)\rangle.$$
Observe that if $E$ is a translation invariant $(B)$-space of ultradistributions of class $*-\dagger$, then so is $\check{E}$. Clearly $\|T_h\|_{\mathcal{L}(E)}=\|T_{-h}\|_{\mathcal{L}(\check{E})}$. Hence the convolution can be defined in the same way as a mapping from $\check{E}'\times E$ to $C(\RR^d)$. We end this section with a simple proposition describing the mapping properties of this convolution. As usual, $L^{\infty}_{\omega}$, the dual of the Beurling algebra $L^{1}_{\omega}$, is the $(B)$-space of all measurable functions satisfying
$$
\|u\|_{\infty,\omega}=\operatorname*{ess}\sup_{x\in\mathbb{R}^{d}} \frac{|g(x)|}{\omega(x)}<\infty.
$$
We need the following two closed subspaces of $L^{\infty}_{\omega}$,
\begin{equation}
\label{UC}
 UC_{\omega}:= \left\{u\in L^{\infty}_{\omega}\big|\, \lim_{h\to0}\|T_{h}u-u\|_{\infty,\omega}=0 \right\}
\end{equation}
and
\begin{equation}
\label{C}
C_{\omega}:= \left\{u\in C(\mathbb{R}^{d})\Big|\, \lim_{|x|\to\infty}\frac{u(x)}{\omega(x)}=0 \right\}.
\end{equation}

The proof of the following proposition is simple and we thus omit it (the second part about the reflexive case follows from Proposition \ref{prop3.3}).

\begin{proposition}\label{convolution E E'}
$E'\ast \check{E} \subseteq  UC_{\omega}$ and $E'\times \check{E} \to  UC_{\omega}$ is continuous. If $E$ is reflexive, then $E'\ast \check{E} \subseteq C_{\omega}$. Similarly $\check{E}'\ast E \subseteq  UC_{\check{\omega}}$ and $\check{E}'\times E \to  UC_{\check{\omega}}$ is continuous. When $E$ is reflexive, $E'\ast \check{E} \subseteq C_{\omega}$ and $\check{E}'\ast E \subseteq C_{\check{\omega}}$.
\end{proposition}

We conclude this section with some examples of translation-invariant $(B)$-spaces of quasianalytic ultradistributions.

\begin{example}[Weighted $L^{p}_{\eta}$ spaces]\label{example}
Let $\eta$ be an \emph{ultrapolynomially bounded weight function of class $\dagger$}, that is, a (Borel) measurable function $\eta:\mathbb{R}^d\rightarrow (0,\infty)$ that fulfills the requirement $\eta(x+h)\leq C\eta(x)e^{A(\tau|h|)}$ for some $C,\tau>0$ (for every $\tau>0$ there exists $C>0$). For $1\leq p<\infty$ we denote as $L^p_{\eta}$ the spaces of measurable functions $g$ such that $\|g|\|_{p,\eta}:=\|\eta g\|_{p}<\infty$. Clearly $L^p_{\eta}$ are translation-invariant $(B)$-spaces of ultradistributions of class $*-\dagger$ for $p\in[1,\infty)$ and for any sequence $(M_{p})_{p\in\mathbb{N}}$. On the other hand, we make an exception and define $L^\infty_{\eta}$ via the norm $\|g\|_{\infty,\eta}:=\| g/\eta\|_{\infty}$. We also introduce the closed spaces $ UC_{\eta}$ and $C_{\eta}$ of $L^{\infty}_{\eta}$ as in (\ref{UC}) and (\ref{C}) with $\omega$ replaced by $\eta$. Note that $C_{\eta}$ is a translation-invariant $(B)$-spaces of ultradistributions of class $*-\dagger$ because $\mathcal{S}^*_{\dagger}(\mathbb{R}^{d})$ is dense in it, while $L^{\infty}_{\eta}$ and $UC_{\eta}$ fail to have this property.

As usual, we write $q$ for the conjugate index of $p$. As well known, $(L^{p}_{\eta})'=L^{q}_{\eta^{-1}}$ if $1< p<\infty$ and $(L^{1}_{\eta})'=L^{\infty}_{\eta}$. In view of Proposition \ref{thgoodspace}, the space $E'_{\ast}$ corresponding to $E=L^{p}_{\eta^{-1}}$ is $E'_{\ast}=L^{q}_{\eta}$ whenever $1<p<\infty$. On the other hand, Proposition \ref{prop3.11} gives
that $E'_{\ast}= UC_{\eta}$ for  $E=L^{1}_{\eta}$. The Beurling algebra of $L^{p}_{\eta}$ can be explicitly determined as in \cite[Proposition 10]{DPV}, we state the result for the reader's convenience. Note that when $\log \eta$ is a subadditive function and $\eta(0)=1$, the following proposition yields $\omega_{\eta}=\eta$ (a.e.).

\begin{proposition}\label{weight prop}
Let $\omega_{\eta}(h):= \operatorname*{ess}\sup_{x\in\mathbb{R}^{d}} \eta(x+h)/\eta(x)$. Then
$$\|T_{-h}\|_{\mathcal{L}(L^{p}_{\eta})}=
\begin{cases}
\omega_{\eta}(h) & \mbox{ if } p\in[1,\infty),
\\ \omega_{\eta}(-h)
  & \mbox{ if } p=\infty.
\end{cases}
$$
Consequently, the Beurling algebra of $L_{\eta}^{p}$ is $L_{\omega_{\eta}}^{1}$ if $p=[1,\infty)$ and $L^{1}_{\check{\omega}_{\eta}}$ if $p=\infty$.
\end{proposition}

Clearly, the Beurling algebra of $C_{\eta}$ is $L^1_{\check{\omega}_{\eta}}$. We now compute the space $E'_{\ast}$ corresponding to $E=C_{\eta}$. Note that $\eta$ can be assumed to be continuous (the continuous weight $\eta_{1}=\eta\ast \varphi$ defines an equivalent norm if we choose $\varphi\in\mathcal{D}(\mathbb{R}^{d})$ being non-negative with $\int_{\mathbb{R}^{d}}\varphi(x)dx=1$). Thus $E=C_{\eta}$ is isometrically isomorphic to $C_0(\RR^d)$, the isometry being $J_{\eta}:C_{\eta}\rightarrow C_0(\RR^d)$, $J_{\eta}(\psi)=\psi/\eta$. Hence ${}^tJ_{\eta}:\mathcal{M}^1\rightarrow \left(C_{\eta}\right)'$ is isometric isomorphism and thus for each $f\in\left(C_{\eta}\right)'$ there exists a unique finite measure $\nu\in\mathcal{M}^1$ such that $\langle f,\psi\rangle=\ds\int_{\RR^d}\psi(x)/\eta(x) d\mu(x)$ for all $\psi\in C_{\eta}$. We will denote the dual of $C_{\eta}$ by $\mathcal{M}^1_{\eta}$.  Now, one easily verifies that $L^1_{\omega_{\eta}}*\mathcal{M}^1_{\eta}\subseteq L^1_{\eta}$ and since $\SSS^*_{\dagger}(\RR^d)$ is dense in $L^1_{\eta}$ (the proof is analogous to that of Lemma \ref{densitySL}) and $\SSS^*_{\dagger}(\RR^d)*\SSS^*_{\dagger}(\RR^d)$ is dense in $\SSS^*_{\dagger}(\RR^d)$ (cf. Remark \ref{r11}), we obtain that $E'_*=L^1_{\eta}$.

\section{Ultradistribution spaces of class $*-\dagger$ associated to translation-invariant $(B)$-spaces}\label{new spaces}

In this section we construct and study test function and ultradistribution spaces associated to translation-invariant $(B)$-spaces of ultradistributions of class $*-\dagger$. The construction of such spaces is similar to the one given in \cite{DPPV} in the non-quasianalytic case; however, the study of their properties requires new non-trivial arguments. We recall that throughout the rest of the paper $E$ stands for a tempered translation-invariant $(B)$-space of ultradistributions whose growth function of its translation group is $\omega$ (cf. (\ref{ravenstvo1})). The $(B)$-space $E'_{\ast}\subseteq E'$ was introduced in Definition \ref{goodspace}.

\end{example}

\subsection{The test function space $\mathcal{D}^*_{E}$}\label{tspace}

We begin by constructing our test space. Let $$\mathcal{D}^{M_p,m}_{E}=\left\{\varphi\in E\Big|\, D^{\alpha}\varphi\in E, \forall \alpha\in\NN^d, \|\varphi\|_{E,m}=\sup_{\alpha\in\NN^d}\frac{m^\alpha \|D^\alpha\varphi\|_E}{M_\alpha}<\infty\right\}.$$ It is easy to verify that $\mathcal{D}^{M_p,m}_{E}$ is $(B)$-space with the norm $\|\cdot\|_{E,m}$. None of these spaces is trivial. To see this in the Beurling case one only needs to use the continuity of the inclusion $\SSS^{(M_p)}_{(A_p)}(\RR^d)\rightarrow E$ to obtain that $\SSS^{(M_p)}_{(A_p)}(\RR^d)\subseteq \mathcal{D}^{M_p,m}_{E}$ for each $m>0$. In the Roumieu case observe that $\SSS^{(M_p)}_{(A_p)}(\RR^d)$ is continuously injected into $\SSS^{\{M_p\}}_{\{A_p\}}(\RR^d)$, hence we have the continuous inclusions $\SSS^{(M_p)}_{(A_p)}(\RR^d)\rightarrow E$. Now, similarly one proves that $\SSS^{(M_p)}_{(A_p)}(\RR^d)\subseteq \mathcal{D}^{M_p,m}_{E}$ for each $m>0$. Obviously, $\mathcal{D}^{M_p,m_1}_{E}\subseteq\mathcal{D}^{M_p,m_2}_{E}$ for $m_2<m_1$ and the inclusion mapping is continuous. As l.c.s., we define
$$\mathcal{D}_E^{(M_p)}=\lim_{\substack{\longleftarrow\\ m\rightarrow\infty}} \mathcal{D}_E^{M_p,m},\,\,\,\mathcal{D}_E^{\{M_p\}}=\lim_{\substack{\longrightarrow\\ m\rightarrow 0}} \mathcal{D}_E^{M_p,m}.$$
Since $\mathcal{D}_E^{\{M_p\},m}$ is continuously injected into $E$ for each $m>0$, $\mathcal{D}_E^{\{M_p\}}$ is indeed a (Hausdorff) l.c.s.. Moreover $\mathcal{D}_E^{\{M_p\}}$ is a barreled bornological $(DF)$-space, since it is an inductive limit of $(B)$-spaces. Obviously $\mathcal{D}_E^{(M_p)}$ is an $(F)$-space. Of course $\mathcal{D}_E^{(M_p)}$ and $\mathcal{D}_E^{\{M_p\}}$ are continuously injected into $E$.

 Additionally, in the Roumieu case, for each fixed $(r_p)\in\mathfrak{R}$ we define the $(B)$-space
\beqs
\mathcal{D}_E^{\{M_p\},(r_p)}=\left\{\varphi \in E\Bigg|\, D^{\alpha}\varphi\in E, \forall\alpha\in\NN^d, \|\varphi\|_{E,(r_p)}=\sup_{\alpha}\frac{\|D^\alpha\varphi\|_E}{M_\alpha \prod_{j=1}^{|\alpha|}r_j }<\infty\right\},
\eeqs
with norm $\|\cdot\|_{E,(r_p)}$. Since for $k>0$ and $(r_p)\in\mathfrak{R}$, there exists $C>0$ such that $k^{|\alpha|}\geq C/\left(\prod_{j=1}^{|\alpha|}r_j\right)$, $\DD^{\{M_p\},k}_E$ is continuously injected into $\DD^{\{M_p\},(r_p)}_E$. Define as l.c.s. $\ds\tilde{\mathcal{D}}^{\{M_p\}}_E=\lim_{\substack{\longleftarrow\\ (r_p)\in \mathfrak{R}}}\mathcal{D}_E^{\{M_p\},(r_p)}$. Then $\tilde{\DD}^{\{M_p\}}_E$ is a complete l.c.s. and $\DD^{\{M_p\}}_E$ is continuously injected into $\tilde{\DD}^{\{M_p\}}_E$.

\begin{lemma}\label{regular}
The space $\mathcal{D}_E^{\{M_p\}}$ is regular, i.e., every bounded set $B$ in $\mathcal{D}_E^{\{M_p\}}$ is bounded in some $\mathcal{D}_E^{\{M_p\},m}$. In addition $\DD^{\{M_p\}}_E$ is complete.
\end{lemma}

\begin{proof}
An adaptation of the proof of \cite[Proposition 5.1]{DPPV} proves the lemma.
\end{proof}

Similarly as in the first part of the proof of \cite[Proposition 5.1]{DPPV} one can prove, by using \cite[Lemma 3.4]{Komatsu3}, that $\DD^{\{M_p\}}_E$ and $\tilde{\DD}^{\{M_p\}}_E$ are equal as sets, i.e., the canonical inclusion $\DD^{\{M_p\}}_E\rightarrow \tilde{\DD}^{\{M_p\}}_E$ is surjective.

 The next proposition gives the relationship between $\SSS^*_{\dagger}(\RR^d)$, $\DD^*_E$ and $E$. The proof is essentially the same as that of \cite[Proposition 5.2]{DPPV}.

\begin{proposition}\label{970}
The following dense inclusions hold $\mathcal{S}^*_{\dagger}(\mathbb{R}^d)\hookrightarrow\mathcal{D}^{*}_E\hookrightarrow E\hookrightarrow \mathcal{S}'^*_{\dagger}(\mathbb{R}^d)$ and $\mathcal{D}^{*}_E$ is a topological module over the Beurling algebra $L^{1}_{\omega}$, i.e., the convolution $*:L^1_{\omega}\times\DD^*_E\rightarrow \DD^*_E$ is continuous. Moreover, in the Beurling case the following estimate
\begin{equation}\label{eq5}
\|u\ast \varphi\|_{E,m}\leq \|u\|_{1,\omega}\| \varphi\|_{E,m},\,\,\, m>0
\end{equation}
holds. In the Roumieu case, for each $m>0$ the convolution is also continuous bilinear mapping $L^1_{\omega}\times\DD^{M_p,m}_E\rightarrow \DD^{M_p,m}_E$ and the inequality (\ref{eq5}) holds.
\end{proposition}

We will often use the following results on the action of ultradifferential operators on the test space $\DD^*_E$ (see \cite{DPPV} for their proofs).
\begin{lemma}\label{ultradiff}
If $P(D)$ is ultradifferential operator of $*$ type, then $P(D):\mathcal{D}^*_E\rightarrow \mathcal{D}^*_E$ is continuous.
\end{lemma}

\begin{lemma}\label{udozc}
Every ultradifferential operator $P(D)$ of $\{M_p\}$ class acts continuously on $\tilde{\DD}^{\{M_p\}}_E$.
\end{lemma}

It turns out that all elements of our test function space $\mathcal{D}^*_{E}$ are ultradifferentiable functions of class *. We need the following lemmas in order to establish this fact.

\begin{lemma}\label{lemma4.2}
There exists $l>0$ (there exists $(l_p)\in \mathfrak{R}$) such that $\SSS^{M_p,l}_{A_p,l}\subseteq E\cap E'_{\ast}$ ($\SSS^{M_p,(l_p)}_{A_p,(l_p)}\subseteq E\cap E'_{\ast}$). Moreover, the inclusion mappings $\SSS^{M_p,l}_{A_p,l}\to E$ and $\SSS^{M_p,l}_{A_p,l}\to E'_{*}$ ($\SSS^{M_p,(l_p)}_{A_p,(l_p)}\to E$ and $\SSS^{M_p,(l_p)}_{A_p,(l_p)}\to E'_{\ast}$) are continuous.
\end{lemma}

\begin{proof} We give the proof in the Roumieu case, the Beurling case is similar. Since the inclusion $\SSS^{\{M_p\}}_{\{A_p\}}(\RR^d)\rightarrow E$ is continuous and $\ds \SSS^{\{M_p\}}_{\{A_p\}}(\RR^d)=\lim_{\substack{\longleftarrow\\ (r_p)\in\mathfrak{R}}}\SSS^{M_p,(r_p)}_{A_p,(r_p)}$ there exist $C>0$ and $(r_p)\in\mathfrak{R}$ such that $\|\varphi\|_E\leq C \sigma_{(r_p)}(\varphi)$, $\forall \varphi\in \SSS^{\{M_p\}}_{\{A_p\}}(\RR^d)$. For this $(r_p)$, by Lemma \ref{appincl}, there exist $(k_p)\in\mathfrak{R}$ and $\chi_n,\varphi_n\in\SSS^{\{M_p\}}_{\{A_p\}}(\RR^d)$, $n\in\ZZ_+$, such that $\chi_n*(\varphi_n\psi)\in\SSS^{\{M_p\}}_{\{A_p\}}(\RR^d)$ for each $n\in\ZZ_+$ and $\chi_n*(\varphi_n\psi)\rightarrow \psi$ when $n\rightarrow\infty$ in $\SSS^{M_p,(r_p)}_{A_p,(r_p)}$ for all $\psi\in \SSS^{M_p,(k_p)}_{A_p,(k_p)}$. We have
\beq\label{1131}
\|\chi_n*(\varphi_n\psi)\|_E\leq C \sigma_{(r_p)}(\chi_n*(\varphi_n\psi)).
\eeq
We obtain that $\chi_n*(\varphi_n\psi)$ is a Cauchy sequence in $E$, hence it converges. Since $\chi_n*(\varphi_n\psi)\rightarrow \psi$ in $\SSS^{M_p,(r_p)}_{A_p,(r_p)}$ the convergence also holds in $\SSS'^{\{M_p\}}_{\{A_p\}}(\RR^d)$. But $E$ is continuously injected into $\SSS'^{\{M_p\}}_{\{A_p\}}(\RR^d)$ thus the limit of $\chi_n*(\varphi_n\psi)$ in $E$ must be $\psi$. If we let $n\rightarrow\infty$ in (\ref{1131}) we have $\|\psi\|_E\leq C\sigma_{(r_p)}(\psi)\leq C\sigma_{(k_p)}(\psi)$, which gives the desired continuity of the inclusion $\SSS^{M_p,(k_p)}_{A_p,(k_p)}\to E$. Similarly, one obtains the continuous inclusion $\SSS^{M_p,(k'_p)}_{A_p,(k'_p)}\to E'_{\ast}$ possibly with another $(k'_p)\in\mathfrak{R}$. The conclusion of the lemma now follows by taking $(l_p)\in\mathfrak{R}$ defined as $l_p=\min\{k_p,k'_p\}$, $p\in\ZZ_+$.
\end{proof}

\begin{lemma}\label{smolemm}
Let $f\in\SSS'^*_{\dagger}(\RR^d)$ be a continuous function such that for each $\beta\in\NN^d$ the ultradistributional derivative $D^{\beta}f$ is a continuous function with ultrapolynomial growth of type $\dagger$. Then $f\in C^{\infty}(\RR^d)$.
\end{lemma}

\begin{proof} Since $f$ is continuous $f\in\DD'(\RR^d)$ (the Schwartz space of distributions). First we prove that the ultradistributional derivatives of $f$ coincide with its distributional derivatives. We give the proof in the Roumieu case. The Beurling case is similar. Let $\beta\in\NN^d$. Denote by $f_{\beta}$ the distributional derivative $D^{\beta}f$ of $f$ and by $\tilde{f}_{\beta}$ the ultradistributional derivative $D^{\beta}f$ of $f$. Since $f$ and $\tilde{f}_{\beta}$ are continuous functions of ultrapolynomial growth $\{A_p\}$, similarly as in the proof of $(c)\Leftrightarrow(\tilde{c})$ in \cite[Theorem 4.2]{DPPV}, one can prove that there exist $(r_p)\in\mathfrak{R}$ and $C>0$ such that $|f(x)|\leq C e^{B_{r_p}(|x|)}$ and $|\tilde{f}_{\beta}(x)|\leq C e^{B_{r_p}(|x|)}$. Pick $(k_p)\in\mathfrak{R}$ such that $(k_p)\leq (r_p)$ and $e^{B_{r_p}(|\cdot|)}e^{-B_{k_p}(|\cdot|)}\in L^1(\RR^d)$. Fix $\psi\in\DD(\RR^d)$. Let $\chi_n\in\SSS^{\{M_p\}}_{\{A_p\}}(\RR^d)$, $n\in\ZZ_+$, be defined as in $ii)$ of Lemma \ref{appincl}. One easily verifies that $\psi_n=\chi_n*\psi\in\SSS^{\{M_p\}}_{\{A_p\}}(\RR^d)$. Let $\alpha\leq \beta$. Observe that
\begin{align} \label{13efd}
& e^{B_{k_p}(|x|)}\left|D^{\alpha}\psi_n(x)-D^{\alpha}\psi(x)\right|\\
&\quad
\leq 2\int_{\RR^d}|\chi(y)|e^{B_{k_p}(2|y|)}\left|D^{\alpha}\psi(x-y/n)-D^{\alpha}\psi(x)\right|e^{B_{k_p}(2|x-y/n|)}dy. \nonumber
\end{align}
Let $\varepsilon>0$. Since $\psi$ is compactly supported,
\beqs
\left|D^{\alpha}\psi(x-y/n)-D^{\alpha}\psi(x)\right|e^{B_{k_p}(2|x-y/n|)}&\leq& C_1+\left|D^{\alpha}\psi(x)\right|e^{B_{k_p}(2|x-y/n|)}\\
&\leq& C_1+C_2e^{B_{k_p}(4|y|)}.
\eeqs
As $|\chi(y)|e^{B_{k_p}(2|y|)}(C_1+C_2e^{B_{k_p}(4|y|)})\in L^1(\RR^d)$, there exists $c_1\geq 1$ such that
\beqs
\int_{|y|\geq c_1}|\chi(y)|e^{B_{k_p}(2|y|)}(C_1+C_2e^{B_{k_p}(4|y|)})dy\leq \varepsilon/4.
\eeqs
Of course, we can assume that $c_1$ is large enough such that $\mathrm{supp}\, \psi\subseteq \{x\in\RR^d|\, |x|\leq c_1\}$. Clearly $D^{\alpha}\psi(x)=0$ and $D^{\alpha}\psi(x-y/n)=0$ for all $n\in\ZZ_+$ when $|x|> 2c_1$ and $|y|\leq c_1$. Hence, for $|x|\leq 2 c_1$, $|y|\leq c_1$ and $n\in\ZZ_+$ there exists $C_2$ such that $e^{B_{k_p}(2|x-y/n|)}\leq C_2$. Since $D^{\alpha}\psi$ is continuous, there exists $n_0\in\ZZ_+$ such that for all $n\geq n_0$, $|x|\leq 2 c_1$ and $|y|\leq c_1$
\beqs
\left|D^{\alpha}\psi(x-y/n)-D^{\alpha}\psi(x)\right|\leq \varepsilon/\left(4C_2 \left\|\chi e^{B_{k_p}(2|\cdot|)}\right\|_{L^1(\RR^d)}\right).
\eeqs
These estimates, together with (\ref{13efd}), imply $\left\|e^{B_{k_p}(|\cdot|)}\left(D^{\alpha}\psi_n-D^{\alpha}\psi\right)\right\|_{L^{\infty}(\RR^d)}\leq \varepsilon$ for all $n\geq n_0$. We obtain that for each $\alpha\leq \beta$, $e^{B_{k_p}(|\cdot|)}D^{\alpha}\psi_n\rightarrow e^{B_{k_p}(|\cdot|)}D^{\alpha}\psi$ in $L^{\infty}(\RR^d)$. Now, dominated convergence implies
\beqs
\lim_{n\rightarrow \infty}\int_{\RR^d}\tilde{f}_{\beta}(x)\psi_n(x)dx&=&\int_{\RR^d}\tilde{f}_{\beta}(x)\psi(x)dx,\\
\lim_{n\rightarrow \infty} \int_{\RR^d}f(x)(-D)^{\beta}\psi_n(x)dx&=&\int_{\RR^d}f(x)(-D)^{\beta}\psi(x)dx.
\eeqs
Hence
\beqs
\langle \tilde{f}_{\beta},\psi\rangle=\lim_{n\rightarrow \infty} \int_{\RR^d}\tilde{f}_{\beta}(x)\psi_n(x)dx=\lim_{n\rightarrow \infty} \int_{\RR^d}f(x)(-D)^{\beta}\psi_n(x)dx=\langle f_{\beta},\psi\rangle.
\eeqs
Since $\psi\in \DD(\RR^d)$ is arbitrary $\tilde{f}_{\beta}= f_{\beta}$. In other words $f$ is a continuous function whose all distributional derivatives are continuous functions. Now the Sobolev imbedding theorem applied on a ball with center at a fixed point $x\in\RR^d$ implies that $f$ is $C^{\infty}$ in that ball. As $x$ is arbitrary, the assertion follows.
\end{proof}

Define for every $m,h>0$ the $(B)$-spaces $$\mathcal{O}^{M_p}_{A_p,m,h}=\left\{\varphi\in C^{\infty}(\RR^d)\bigg|\,\|\varphi\|_{m,h}=\left(\sum_{\alpha\in \mathbb{N}^d}\frac{m^{2|\alpha|}}{M_{\alpha}^2}\left\|D^{\alpha}\varphi e^{-A(h|\cdot|)}\right\|^2_{L^2}\right)^{1/2}<\infty\right\}.$$ Observe that for $m_1\leq m_2$ we have the continuous inclusion $\mathcal{O}^{M_p}_{A_p,m_2,h}\rightarrow\mathcal{O}^{M_p}_{A_p,m_1,h}$ and for $h_1\leq h_2$ the inclusion $\mathcal{O}^{M_p}_{A_p,m,h_1}\rightarrow\mathcal{O}^{M_p}_{A_p,m,h_2}$ is also continuous. As l.c.s. we define
\beqs
\mathcal{O}_{(A_p),h}^{(M_p)}=\lim_{\substack{\longleftarrow\\ m\rightarrow \infty}}\mathcal{O}^{M_p}_{A_p,m,h}&,&\,\, \mathcal{O}_{(A_p),C}^{(M_p)}=\lim_{\substack{\longrightarrow\\ h\rightarrow \infty}}\mathcal{O}^{(M_p)}_{(A_p),h};\\
\mathcal{O}_{\{A_p\},h}^{\{M_p\}}=\lim_{\substack{\longrightarrow\\ m\rightarrow 0}}\mathcal{O}^{M_p}_{A_p,m,h}&,& \,\,\mathcal{O}_{\{A_p\},C}^{\{M_p\}}=\lim_{\substack{\longleftarrow\\ h\rightarrow 0}}\mathcal{O}^{\{M_p\}}_{\{A_p\},h}.
\eeqs
Observe that $\mathcal{O}_{(A_p),h}^{(M_p)}$ is an $(F)$-space and since all inclusions $\mathcal{O}_{(A_p),h}^{(M_p)}\rightarrow C^{\infty}(\RR^d)$ are continuous (by the Sobolev imbedding theorem), $\mathcal{O}_{(A_p),C}^{(M_p)}$ is indeed a (Hausdorff) l.c.s.. Moreover, as an inductive limit of barreled and bornological spaces $\mathcal{O}_{(A_p),C}^{(M_p)}$ is barreled and bornological. Also $\mathcal{O}_{\{A_p\},h}^{\{M_p\}}$ is (Hausdorff) l.c.s., because all inclusions $\mathcal{O}^{M_p}_{A_p,m,h}\rightarrow C^{\infty}(\RR^d)$ are continuous (by the Sobolev imbedding theorem). Hence $\mathcal{O}_{\{A_p\},C}^{\{M_p\}}$ is indeed a (Hausdorff) l.c.s.. Furthermore, $\mathcal{O}_{\{A_p\},h}^{\{M_p\}}$ is a barreled and bornological $(DF)$-space, as inductive limit of $(B)$-spaces. By this considerations it also follows that $\mathcal{O}_{\dagger,C}^*$ is continuously injected into $C^{\infty}(\RR^d)$. One easily verifies that $\SSS^*_{\dagger}(\RR^d)$ is continuously and densely injected into $\mathcal{O}_{\dagger,C}^*$. We mention that $\mathcal{O}_{\dagger,C}^*$ was introduced and studied in \cite{DPPV} in the non-quasianalytic case.

\begin{proposition}\label{smooth prop}
The embedding $\mathcal{D}^*_{E}\hookrightarrow \mathcal{O}^*_{\dagger,C}(\mathbb{R}^{d})$ holds. Furthermore, for $\varphi\in\mathcal{D}^*_{E}$, $D^{\alpha}\varphi\in C_{\check{\omega}}$ for all $\alpha\in\NN^d$ and they satisfy the following growth condition: For every $m>0$ (for some $m>0$)
\begin{equation}\label{eqgrowth}
\sup_{\alpha\in\NN^d}\frac{m^{|\alpha|}}{M_{\alpha}} \left\|D^{\alpha}\varphi\right\|_{L^{\infty}_{\check{\omega}}{(\RR^d)}}<\infty.
\end{equation}
\end{proposition}

\begin{proof} Let $r>0$ ($(r_p)\in\mathfrak{R}$) be as in Lemma \ref{lemma4.2}, that is, $\SSS^{M_p,r}_{A_p,r}\subseteq E\cap E'_{\ast}$ ($\SSS^{M_p,(r_p)}_{A_p,(r_p)}\subseteq E\cap E'_{\ast}$) and the inclusion mappings $\SSS^{M_p,r}_{A_p,r}\to E$ and $\SSS^{M_p,r}_{A_p,r}\to E'_{*}$ ($\SSS^{M_p,(r_p)}_{A_p,(r_p)}\to E$ and $\SSS^{M_p,(r_p)}_{A_p,(r_p)}\to E'_{\ast}$) are continuous. By Proposition \ref{parametrix}, there exist $u\in\SSS^{M_p,r}_{A_p,r}$ and $P(D)$ of type $(M_p)$ ($u\in\SSS^{M_p,(r_p)}_{A_p,(r_p)}$ and $P(D)$ of type $\{M_p\}$) such that $P(D)u=\delta$. Let $f\in\mathcal{D}^*_{E}$. Then $f=\left(P(D)u\right)*f$. We first prove that
\beq\label{equ11}
f=\left(P(D)u\right)*f=P(D)(u*f)=u*\left(P(D)f\right).
\eeq
Since $\check{u}\in \SSS^{M_p,r}_{A_p,r}\subseteq E'$ ($\check{u}\in \SSS^{M_p,(r_p)}_{A_p,(r_p)}\subseteq E'$) and $f\in\DD^*_E\subseteq E$, Proposition \ref{convolution E E'} implies that $u*f\in  UC_{\check{\omega}}\subseteq \SSS'^*_{\dagger}(\RR^d)$, hence $P(D)(u*f)$ is well a defined element of $\SSS'^*_{\dagger}(\RR^d)$. Similarly, by Lemma \ref{ultradiff}, $P(D)f\in \DD^*_E\subseteq E$, hence Proposition \ref{convolution E E'} implies $u*\left(P(D)f\right)$ is well defined element of $\SSS'^*_{\dagger}(\RR^d)$. By Proposition \ref{970} there exists a net $f_{\nu}\in \SSS^*_{\dagger}(\RR^d)$ which converges to $f$ in $\DD^*_E$. Then
\beq\label{11sqc}
f_{\nu}=\delta*f_{\nu}= \left(P(D)u\right)*f_{\nu}=P(D)(u*f_{\nu})=u*\left(P(D)f_{\nu}\right).
\eeq
Now, since $f_{\nu}\rightarrow f$ in $\DD^*_E$ the convergence also holds in $E$ and thus Proposition \ref{convolution E E'} implies $u*f_{\nu}\rightarrow u*f$ in ${UC}_{\check{\omega}}$ and therefore also in $\SSS'^*_{\dagger}(\RR^d)$. Hence $P(D)(u*f_{\nu})\rightarrow P(D)(u*f)$ in $\SSS'^*_{\dagger}(\RR^d)$. Next $P(D)f_{\nu}\rightarrow P(D)f$ in $\DD^*_E$ (cf. Lemma \ref{ultradiff}) consequently also in $E$. Again, Proposition \ref{convolution E E'} implies $u*\left(P(D)f_{\nu}\right)\rightarrow u*\left(P(D)f\right)$ in ${UC}_{\check{\omega}}$, hence also in $\SSS'^*_{\dagger}(\RR^d)$. Now after taking limit in (\ref{11sqc}), we obtain (\ref{equ11}). For $\beta\in\NN^d$, since $D^{\beta}f\in\DD^*_E$, (\ref{equ11}) implies $D^{\beta}f=u*D^{\beta}P(D)f$. Since $D^{\beta}P(D)f\in\DD^*_E\subseteq E$, Proposition \ref{convolution E E'} and the discussion preceding it imply that $D^{\beta}f$ is continuous function and $D^{\beta}f\in  UC_{\check{\omega}}$ for each $\beta\in\NN^d$. Thus, Lemma \ref{smolemm} implies that $f\in C^{\infty}(\RR^d)$. To prove the inclusion $\mathcal{D}^*_{E}\rightarrow \mathcal{O}^*_{\dagger,C}(\mathbb{R}^{d})$, we consider first the $(M_p)$ case. Let $m>0$ be arbitrary but fixed. Since $P(D)=\sum_{\alpha}c_{\alpha}D^{\alpha}$ is of $(M_p)$ type, there exist $m_1,C'>0$ such that $|c_{\alpha}|\leq C'm_1^{|\alpha|}/M_{\alpha}$. Let $m_2=4\max\{m,m_1\}$. By Lemma \ref{ultradiff} (and its proof), we have
\beqs
|D^{\beta}f(x)|\leq \|u\|_{\check{E}'}\left\|D^{\beta}P(D)f(x)\right\|_E\omega(-x)\leq C_2\omega(-x)\|\check{u}\|_{E'}\|f\|_{E,m_2H}\frac{M_{\beta}}{(2m)^{|\beta|}}.
\eeqs
Hence
\begin{equation}\label{bounds}
\frac{(2m)^{|\beta|}\left|D^{\beta}f(x)\right|}{M_{\beta}w(-x)}\leq C''\|\check{u}\|_{E'}\|f\|_{E,m_2H}.
\end{equation}
Since there exist $\tau,C'''>0$ such that $\omega(x)\leq C''' e^{A(\tau|x|)}$, by using \cite[Proposition 3.6]{Komatsu1}, we obtain $\omega(-x)e^{A(\tau|x|)}\leq C_4 e^{A(\tau H|x|)}$. Hence
\beqs
\left(\sum_{\alpha}\frac{m^{2|\alpha|}}{M_{\alpha}^2}\left\|D^{\alpha}f e^{-A(\tau H|\cdot|)}\right\|^2_{L^2}\right)^{1/2}&\leq& C_5\left(\sum_{\alpha}\frac{m^{2|\alpha|}}{M_{\alpha}^2}\left\|\frac{D^{\alpha}f} {\omega(-\cdot)}\right\|^2_{L^{\infty}}\right)^{1/2}\\
&\leq& C\|\check{u}\|_{E'}\|f\|_{E,m_2H},
\eeqs
which proves the continuity of the inclusion $\mathcal{D}^{(M_p)}_{E}\rightarrow \mathcal{O}_{(A_p),\tau H}^{(M_p)}$ and hence also the continuity of the inclusion $\mathcal{D}^{(M_p)}_{E}\rightarrow \mathcal{O}_{(A_p),C}^{(M_p)}$.

In order to prove that the inclusion $\mathcal{D}^{\{M_p\}}_{E}\rightarrow \mathcal{O}_{\{A_p\},C}^{\{M_p\}}$ is continuous, it is enough to prove that for each $h>0$, $\mathcal{D}^{\{M_p\}}_{E}\rightarrow \mathcal{O}_{\{A_p\},h}^{\{M_p\}}$ is continuous. And in order to prove this it is enough to prove that for every $m>0$ there exists $m'>0$ such that we have the continuous inclusion $\mathcal{D}^{M_p,m}_{E}\rightarrow \mathcal{O}_{A_p,m',h}^{M_p}$. So, let $h,m>0$ be arbitrary but fixed. Take $m'\leq m/(4H)$. For $f\in \mathcal{D}^{M_p,m}_{E}$, keeping notations as above, by Lemma \ref{ultradiff} (and its proof), we have
\beqs
|D^{\beta}f(x)|&\leq& \|\check{u}\|_{E'}\left\|D^{\beta}P(D)f(x)\right\|_E\omega(-x)\leq C_2\omega(-x)\|\check{u}\|_{E'}\|f\|_{E,m}\frac{M_{\beta}}{(2m')^{|\beta|}},\\
\eeqs
namely,
\beq\label{boundss}
\frac{(2m')^{|\beta|}\left|D^{\beta}f(x)\right|}{M_{\beta}w(-x)}\leq C''\|\check{u}\|_{E'}\|f\|_{E,m}.
\eeq
For the fixed $h$ take $\tau>0$ such that $\tau H\leq h$. Then there exists $C'''>0$ such that $\omega(x)\leq C''' e^{A(\tau|x|)}$ and by using \cite[Proposition 3.6]{Komatsu1} we obtain $\omega(x)e^{A(\tau|x|)}\leq C_4 e^{A(\tau H|x|)}$. Similarly as above, we have
\beqs
\left(\sum_{\alpha}\frac{m'^{2|\alpha|}}{M_{\alpha}^2}\left\|D^{\alpha}f e^{-A(h|\cdot|)}\right\|^2_{L^2}\right)^{1/2} \leq C\|\check{u}\|_{E'}\|f\|_{E,m},
\eeqs
which proves the continuity of the inclusion $\mathcal{D}^{\{M_p\},m}_{E}\rightarrow \mathcal{O}_{A_p,m',h}^{M_p}$.

Observe that (\ref{eqgrowth}) follows from (\ref{bounds}) and (\ref{boundss}), respectively. It remains to prove that $D^{\alpha}f\in C_{\check{\omega}}$. We will prove this in the Roumieu case as the Beurling case is similar. By using Lemma \ref{udozc}, with a similar technique as above, one can prove that for every $(k_p)\in\mathfrak{R}$ there exists $(l_p)\in\mathfrak{R}$ such that for $f\in\DD^{\{M_p\}}_E$ we have
\beq\label{15557}
\frac{\left|D^{\beta}f(x)\right|}{w(-x)M_{\beta}\prod_{j=1}^{|\beta|}k_j}\leq C''\|\check{u}\|_{E'}\|f\|_{E,(l_p)}.
\eeq
Let $\varepsilon>0$. Since $\SSS^{\{M_p\}}_{\{A_p\}}(\RR^d)$ is dense in $\DD^{\{M_p\}}_E$ (cf. Proposition \ref{970}), it is dense in $\tilde{\DD}^{\{M_p\}}_E$. Pick $\chi\in\SSS^{\{M_p\}}_{\{A_p\}}(\RR^d)$ such that $\|f-\chi\|_{E,(l_p)}\leq \varepsilon/\left(2C''\|\check{u}\|_{E'}\right)$. Since $1=\omega(0)\leq \omega (-x)\omega(x)$, by $(\widetilde{III})$ there exist $(l'_p)\in\mathfrak{R}$ and $C_0>0$ such that $1/\omega(-x)\leq C_0 e^{B_{l'_p}(|x|)}$. Thus, as $\chi\in \SSS^{\{M_p\}}_{\{A_p\}}(\RR^d)$, there exists $K\subset\subset \RR^d$ such that $\ds\frac{\left|D^{\beta}\chi(x)\right|}{w(-x)M_{\beta}\prod_{j=1}^{|\beta|}k_j}\leq \varepsilon/2$ for all $x\in\RR^d\backslash K$ and $\beta\in\NN^d$. Then, by (\ref{15557}), for $x\in\RR^d\backslash K$ and $\beta\in\NN^d$, we have
\beqs
\frac{\left|D^{\beta}f(x)\right|}{w(-x)M_{\beta}\prod_{j=1}^{|\beta|}k_j}\leq \frac{\left|D^{\beta}\left(f(x)-\chi(x)\right)\right|}{w(-x)M_{\beta}\prod_{j=1}^{|\beta|}k_j}+ \frac{\left|D^{\beta}\chi(x)\right|}{w(-x)M_{\beta}\prod_{j=1}^{|\beta|}k_j}\leq \varepsilon,
\eeqs
which proves that $D^{\beta}f\in C_{\check{\omega}}$.
\end{proof}

\begin{remark} If $f\in\SSS^*_{\dagger}(\RR^d)$, the proof of the previous proposition (combined with Proposition \ref{cor:Beurling Algebra}) yields $\|D^{\beta}f\|_E\leq \|u\|_E\|D^{\beta}P(D)f\|_{1,\omega}$, since $u\in E$. Employing a similar technique as in the proof of Lemma \ref{ultradiff} (Lemma \ref{udozc}), we obtain that for every $m>0$ there exist $\tilde{m}>0$ and $C_1>0$ (for every $(k_p)\in \mathfrak{R}$ there exist $(l_p)\in\mathfrak{R}$ and $C_1>0$) such that
\beq\label{111111177}\quad\quad\quad
\|f\|_{E,m}\leq C_1 \sup_{\alpha}\frac{\tilde{m}^{|\alpha|}\left\|D^{\alpha}f\right\|_{1,\omega}}{M_{\alpha}}\quad \bigg( \|f\|_{E,(k_p)}\leq C_1 \sup_{\alpha}\frac{\left\|D^{\alpha}f\right\|_{1,\omega}}{M_{\alpha}\prod_{j=1}^{|\alpha|}l_j}\bigg).
\eeq
\end{remark}

\subsection{The ultradistribution space $\mathcal{D}'^*_{E'_{\ast}}$}\label{subsection DE}

We can now define our new distribution space. We denote by $\mathcal{D}'^*_{E'_{\ast}}$ the strong dual of $\mathcal{D}^*_{E}$. Then, $\mathcal{D}'^{(M_p)}_{E'_{\ast}}$ is a complete $(DF)$-space because $\mathcal{D}^{(M_p)}_{E}$ is an $(F)$-space. Also, $\mathcal{D}'^{\{M_p\}}_{E'_{\ast}}$ is an $(F)$-space as the strong dual of a $(DF)$-space. When $E$ is reflexive, we write $\mathcal{D}'^*_{E'}=\mathcal{D}'^*_{E'_{\ast}}$ in accordance with the last assertion of Theorem \ref{thgoodspace}. The notation $\mathcal{D}'^*_{E'_{\ast}}=(\mathcal{D}^*_{E})'$ is motivated by the next structural theorem which characterizes the elements of this dual space and bounded sets in two ways, in terms of convolution averages and as the product of ultradifferential operators acting on elements of $E'_{\ast}$.

\begin{theorem}\label{karak}
Let $B\subseteq\SSS'^*_{\dagger}(\mathbb{R}^d)$. The following statements are equivalent:
\begin{itemize}
\item [$(i)$] $B$ is a bounded subset of $\mathcal{D}'^*_{E'_{\ast}}$.
\item [$(ii)$] for each $\psi\in\SSS^*_{\dagger}(\RR^d)$, $\{f*\psi|\, f\in B\}$ is a bounded subset of $E'$.
\item [$(iii)$] for each $\psi\in\SSS^*_{\dagger}(\RR^d)$, $\{f*\psi|\, f\in B\}$ is a bounded subset of $E'_*$.
\item [$(iv)$] there exist a bounded subset $B_1$ of $E'$ and an ultradifferential operator $P(D)$ of class $*$ such that each $f\in B$ can be expressed as $f=P(D)g$ with $g\in B_1$.
\item [$(v)$] there exist $B_2\subseteq E'_*\cap  UC_{\omega}$ which is bounded in $E'_*$ and in $ UC_{\omega}$ and an ultradifferential operator $P(D)$ of class $*$ such that each $f\in B$ can be expressed as $f=P(D)g$ with $g\in B_2$. Moreover, if $E$ is reflexive, we may choose $B_2\subseteq E'\cap C_{\omega}$.
\end{itemize}
\end{theorem}

\begin{proof}
We denote $B_{E}=\{\varphi\in\SSS^*_{\dagger}(\mathbb{R}^d)|\, \|\varphi\|_E\leq1\}$.

$(i)\Rightarrow (ii)$. Fix first $\psi\in\SSS^*_{\dagger}(\mathbb{R}^d)$. By Proposition \ref{cor:Beurling Algebra} the set $\check{\psi}\ast B_{E}=\{\check{\psi}*\varphi|\, \varphi \in B_{E}\}$ is bounded in $\mathcal{D}^*_{E}$. As $\DD^*_E$ is barreled, $B$ is equicontinuous. Hence, $|\langle f*\psi,\varphi\rangle|=|\langle f,\check{\psi}*\varphi\rangle|\leq C_{\psi}$, $\forall\varphi\in B_{E}$, $\forall f\in B$. So, $|\langle f*\psi,\varphi\rangle|\leq C_{\psi}\|\varphi\|_E$, $\forall \varphi\in\SSS^*_{\dagger}(\mathbb{R}^d)$, $\forall f\in B$. Since $\SSS^*_{\dagger}(\mathbb{R}^d)$ is dense in $E$, we obtain $\{f*\psi|\, f\in B\}$ is a bounded subset of $E'$, for each $\psi\in \SSS^*_{\dagger}(\mathbb{R}^d)$.

 We prove $(ii)\Rightarrow(iv)$ and $(ii)\Rightarrow(v)$ simultaneously. Let $(ii)$ hold. For arbitrary but fixed $\psi\in\SSS^*_{\dagger}(\mathbb{R}^d)$ we have $\langle f*\check{\varphi},\check{\psi}\rangle=\langle f*\psi,\varphi\rangle$. We obtain that the set $\{\langle f*\check{\varphi},\check{\psi}\rangle|\, \varphi\in B_{E}, f\in B\}$ is bounded in $\mathbb{C}$, i.e., $\{f*\check{\varphi}|\, \varphi\in B_{E}, f\in B\}$ is weakly bounded in $\SSS'^*_{\dagger}(\mathbb{R}^d)$, hence it is equicontinuous. Moreover, Lemma \ref{boundedsetS} implies that $B$ is bounded in $\SSS'^*_{\dagger}(\RR^d)$. We continue the proof in the Roumieu case. The Beurling case is similar. For $(t_p)\in\mathfrak{R}$, denote by $X_{(t_p)}$ the closure of $\SSS^{\{M_p\}}_{\{A_p\}}(\RR^d)$ in $\SSS^{M_p, (t_p)}_{A_p,(t_p)}$. The equicontinuity of the set $\{f*\check{\varphi}|\, \varphi\in B_{E}, f\in B\}$ implies that there exist $(r_p)\in\mathfrak{R}$ and $C>0$ such that
\beq\label{inq1799}
\left|\langle f*\psi,\varphi\rangle\right|\leq C\sigma_{(r_p)}(\psi),\,\, \forall \psi\in \SSS^{\{M_p\}}_{\{A_p\}}(\mathbb{R}^d),\, \forall\varphi\in B_E,\, \forall f\in B.
\eeq
By Lemma \ref{lemma4.2}, there exists $(r'_p)\in\mathfrak{R}$ such that $\SSS^{M_p,(r'_p)}_{A_p,(r'_p)}\subseteq E\cap E'_{\ast}$ and the inclusion mappings $\SSS^{M_p,(r'_p)}_{A_p,(r'_p)}\to E$ and $\SSS^{M_p,(r'_p)}_{A_p,(r'_p)}\to E'_{\ast}$ are continuous. Of course, we can take $(r'_p)\leq (r_p)$. Since $B$ is bounded in $\SSS'^{\{M_p\}}_{\{A_p\}}(\RR^d)$, Proposition \ref{parametrix1} implies that there exist $(l_p),(k_p)\in\mathfrak{R}$ with $(l_p)\leq (k_p)$ such that $f$ can be extended to $X_{(k_p)}$, $\SSS^{M_p,(l_p)}_{A_p,(l_p)}\subseteq X_{(k_p)}$, the convolution is a continuous bilinear mapping from $\SSS^{M_p,(l_p)}_{A_p,(l_p)}\times\SSS^{M_p,(l_p)}_{A_p,(l_p)}$ into $X_{(k_p)}$ and there exists $u\in X_{(l_p)}$ and $P(D)$ of class $\{M_p\}$ such that $P(D)u=\delta$ and $f=P(D)(u*\tilde{f})$, where $\tilde{f}$ is the extension of $f\in B$ to $X_{(k_p)}$ and $u*\tilde{f}$ is the transpose of the continuous mapping $\psi\mapsto \check{u}*\psi$, $\SSS^{\{M_p\}}_{\{A_p\}}(\RR^d)\rightarrow X_{(k_p)}$. We may assume that $(k_p)\leq (r'_p)$. Let $u_n\in\SSS^{\{M_p\}}_{\{A_p\}}(\RR^d)$, $n\in\ZZ_+$, be such that $u_n\rightarrow u$ in $X_{(l_p)}$. The continuity of the convolution $*:\SSS^{M_p,(l_p)}_{A_p,(l_p)}\times\SSS^{M_p,(l_p)}_{A_p,(l_p)}\rightarrow X_{(k_p)}$, together with (\ref{inq1799}), implies
\beqs
\left|\langle u*\tilde{f},\varphi\rangle\right|\leq C',\,\, \forall\varphi\in B_E,\, \forall f\in B,
\eeqs
i.e., $\{u*\tilde{f}|\, f\in B\}$ is a bounded subset of $E'$. Now, $f=P(D)(u*\tilde{f})$ hence $(iv)$ is proved. Since $\{u*\tilde{f}|\, f\in B\}$ is bounded in $E'$, therefore so is it in $\SSS'^{\{M_p\}}_{\{A_p\}}(\RR^d)$, Proposition \ref{parametrix1} again implies that there exist $(l'_p),(k'_p)\in\mathfrak{R}$ with $(l'_p)\leq (k'_p)$ such that $u*\tilde{f}$ can be extended to $X_{(k'_p)}$, $\SSS^{M_p,(l'_p)}_{A_p,(l'_p)}\subseteq X_{(k'_p)}$, the convolution is continuous bilinear mapping from $\SSS^{M_p,(l'_p)}_{A_p,(l'_p)}\times\SSS^{M_p,(l'_p)}_{A_p,(l'_p)}$ into $X_{(k'_p)}$ and there exists $v\in X_{(l'_p)}$ and $P_1(D)$ of class $\{M_p\}$ such that $P_1(D)v=\delta$ and $u*\tilde{f}=P_1(D)(v*(u*\tilde{f}))$, where $v*(u*\tilde{f})$ is the transpose of the continuous mapping $\psi\mapsto \check{v}*\psi$, $\SSS^{\{M_p\}}_{\{A_p\}}(\RR^d)\rightarrow X_{(k'_p)}$. We can suppose that $(k'_p)\leq (l_p)$. Moreover, by Lemma \ref{pppp}, there exist $(t_p)\in\mathfrak{R}$ and $C>0$ such that $\omega(x)\leq C e^{B_{t_p}(|x|)}$ and by Lemma \ref{nwseq} we can assume that $\prod_{j=1}^{p+q}t_j\leq 2^{p+q} \prod_{j=1}^pt_j\cdot \prod_{j=1}^q t_j$, $\forall p,q\in\ZZ_+$. Hence by choosing $(k'_p)\leq (t_p/2H)$, it follows $v\in L^1_{\omega}\cap L^1_{\check{\omega}}$. Now $f=P(D)(u*\tilde{f})=P(D)(P_1(D)(v*(u*\tilde{f})))$. But the composition of two ultradifferential operators is again an ultradifferential operator, hence $f=P_2(D)(v*(u*\tilde{f}))$, where $P_2(D)=P(D)\circ P_1(D)$. Since $v\in L^1_{\omega}\cap L^1_{\check{\omega}}$ and $\{u*\tilde{f}|\, f\in B\}$ is a bounded subset of $E'$, $v* (u*\tilde{f})\in E'_*$, and Corollary \ref{cor3.3} implies that $\{v* (u*\tilde{f})|f\in B\}$ is bounded in $E'_*$. Furthermore, since $\check{v}\in X_{(l'_p)}\subseteq X_{(l_p)}\subseteq \SSS^{M_p,(r'_p)}_{A_p,(r'_p)}\subseteq E$, Proposition \ref{convolution E E'} implies that $\{v* (u*\tilde{f})|f\in B\}$ is a bounded subset of $ UC_{\omega}$ and if $E$ is reflexive, also in $C_{\omega}$. Thus $(v)$ also holds.

 The implications $(iv)\Rightarrow (i)$, $(v)\Rightarrow (i)$, $(iii)\Rightarrow (ii)$ and $(v)\Rightarrow (iii)$ are obvious.
\end{proof}

\begin{proposition}\label{com}
Let $\mathbf{f}:\mathcal{S}^*_{\dagger}(\RR^d)\rightarrow \mathcal{S}'^*_{\dagger}(\RR^d)$ be continuous. The following statements are equivalent:
\begin{itemize}
\item[$i)$] $\mathbf{f}$ commutes with every translation, i.e., $\left\langle \mathbf{f},T_{-h}\varphi\right\rangle = T_{h} \left\langle \mathbf{f},\varphi\right\rangle$, for all $h\in\mathbb{R}^{d}$ and $\varphi\in \mathcal{S}^*_{\dagger}(\mathbb{R}^{d})$.
\item[$ii)$] $\mathbf{f}$ commutes with every convolution, i.e., $\left\langle\mathbf{f},\psi*\varphi\right\rangle=\check{\psi}*\left\langle \mathbf{f},\varphi\right\rangle$, for all $\psi,\varphi\in \mathcal{S}^*_{\dagger}(\mathbb{R}^{d}).$
\item[$iii)$] There exists $f\in \mathcal{S}'^*_{\dagger}(\RR^d)$ such that $\langle \mathbf{f},\varphi\rangle=f*\check{\varphi}$ for every $\varphi\in \mathcal{S}^*_{\dagger}(\RR^d)$.
\end{itemize}
\end{proposition}

\begin{proof}
$i)\Rightarrow ii)$. Let $\varphi,\psi\in\SSS^*_{\dagger}(\RR^d)$. Then $\tilde{\varphi}(x,y)=\varphi(x-y)\psi(y)\in\SSS^*_{\dagger}(\RR^{2d})$. By carefully examining the first part of the proof of Lemma \ref{bocintl}, one can verify that
\beqs
\SSS^*_{\dagger}(\RR^d)\ni L_{\psi,n}(x)=\sum_{t\in D_n}\tilde{\varphi}(x,t)l(n)^{-d}=\sum_{t\in D_n} \varphi(x-t)\psi(t)l(n)^{-d}\rightarrow \psi*\varphi,
\eeqs
in $\SSS^*_{\dagger}(\RR^d)$, where $l(n)$ can be taken to be equal to $n$ (there, the specific definition of $l(n)$ was only needed for the second part of the proof). The continuity of $\mathbf{f}$ implies
\beqs
\langle\mathbf{f},\psi*\varphi\rangle &=&\lim_{n\rightarrow \infty}\left\langle \mathbf{f},\sum_{t\in D_n} \varphi(x-t)\psi(t)n^{-d}\right\rangle\\
&=&\lim_{n\rightarrow \infty}\sum_{t\in D_n} \psi(t)\langle \mathbf{f},T_{-t} \varphi\rangle n^{-d} =\lim_{n\rightarrow\infty}\sum_{t\in D_n} \psi(t) T_t\langle \mathbf{f},\varphi\rangle n^{-d},
\eeqs
in $\SSS'^*_{\dagger}(\RR^d)$. Let $\chi\in \SSS^*_{\dagger}(\RR^d)$. Then
\beqs
\left\langle\lim_{n\rightarrow\infty}\sum_{t\in D_n} \psi(t) T_t\langle \mathbf{f},\varphi\rangle n^{-d},\chi\right\rangle&=& \left\langle\langle\mathbf{f},\varphi\rangle,\lim_{n\rightarrow\infty}\sum_{t\in D_n} \psi(t) T_{-t} \chi n^{-d}\right\rangle\\
&=&\langle\langle\mathbf{f},\varphi\rangle,\psi*\chi\rangle= \langle\check{\psi}*\langle\mathbf{f},\varphi\rangle,\chi\rangle.
\eeqs

 $ii)\Rightarrow iii)$. Take $\chi_n\in\SSS^*_{\dagger}\left(\RR^d\right)$, $n\in\ZZ_+$, to be as in $ii)$ of Lemma \ref{appincl}. Then, for every $\psi\in \mathcal{S}^*_{\dagger}(\RR^d)$ we have that $\psi*\chi_n\rightarrow \psi$ in $\mathcal{S}^*_{\dagger}(\RR^d)$ as $n\rightarrow\infty$ and hence,
\begin{equation}\label{delta}
\check{\psi}*\langle\mathbf{f},\chi_n\rangle=\langle \mathbf{f},\psi*\chi_n\rangle\rightarrow \langle\mathbf{f},\psi\rangle\mbox{ as }n\rightarrow\infty.
\end{equation}
Thus $\{\check{\psi}*\langle\mathbf{f},\chi_n\rangle|\, n\in\ZZ_+\}$ is bounded in $\SSS'^*_{\dagger}(\RR^d)$. Lemma \ref{boundedsetS} implies that $B=\{\langle\mathbf{f},\chi_n\rangle|\, n\in\ZZ_+\}$ is bounded in $\SSS'^*_{\dagger}(\RR^d)$. As $\SSS'^*_{\dagger}(\RR^d)$ is Montel, its closure $\overline{B}$ is compact and the weak and the strong topologies on $\overline{B}$ coincide. As $\overline{B}$ is equicontinuous and $\SSS^*_{\dagger}(\RR^d)$ is separable, the weak topology on $\overline{B}$ is metrizable (cf. \cite[Theorem 4.7, p. 87]{Sch}) hence also the strong topology. Thus, there exists a subsequence $\langle\mathbf{f},\chi_{n_k}\rangle\in B$, $k\in\ZZ_+$, which converges to $f\in \SSS'^*_{\dagger}(\RR^d)$. Now, (\ref{delta}) implies that $\langle\mathbf{f},\psi\rangle=\check{\psi}*f$.

 The implication $iii)\Rightarrow i)$ is clear.
\end{proof}

If $F$ is a l.c.s., as in \cite{Komatsu3}, we define
\beqs
\mathcal{S}'^*_{\dagger}\left(\mathbb{R}^{d},F\right)= \SSS'^*_{\dagger}(\RR^d)\varepsilon F=\mathcal{L}_{\epsilon}\left(\left(\SSS'^*_{\dagger}(\RR^d)\right)'_c, F\right)=\mathcal{L}_b\left(\SSS^*_{\dagger}(\RR^d), F\right),
\eeqs
where the indices $\epsilon$ and $c$ stand for the topology of equicontinuous convergence and the topology of compact convex circled convergence, respectively; the last equality follows from the fact that $\SSS^*_{\dagger}(\RR^d)$ and $\SSS'^*_{\dagger}(\RR^d)$ are complete Montel spaces. If $F$ is complete, since $\SSS'^*_{\dagger}(\RR^d)$ is nuclear, it satisfies the weak approximation property, hence $\SSS'^*_{\dagger}(\RR^d)\varepsilon F\cong \SSS'^*(\RR^d) \hat{\otimes} F$, cf. \cite[Proposition 1.4]{Komatsu3} (for the definition of the $\varepsilon$ tensor product; for the definition of the weak approximation property and their connection we refer to \cite{SchwartzV} and \cite{Komatsu3}).

\begin{corollary}\label{convolution corollary}
Let $\mathbf{f}\in\mathcal{S}'^*_{\dagger}(\mathbb{R}^{d},E'_{\sigma(E',E)})$. If $\mathbf{f}$ commutes with every translation in sense of Proposition \ref{com}, then there exists $f\in \mathcal{D}_{E'_{\ast}}'^*$ such that $\mathbf{f}$ is of the form
\begin{equation}\label{eq:8}
\left\langle\mathbf{f},\varphi\right\rangle=f\ast \check{\varphi}, \ \ \ \varphi\in\mathcal{S}^*_{\dagger}(\mathbb{R}^{d}).
\end{equation}
\end{corollary}
\begin{proof}
The proof is analogous to the proof of \cite[Corollary 6.4]{DPPV}.
\end{proof}

Our results from above implicitly suggest to embed the ultradistribution space $\mathcal{D}'^*_{E'_{\ast}}$ into the space of $E'$-valued ultradistributions as follows. Define first the continuous injection $\iota:\mathcal{S}'^*_{\dagger}(\mathbb{R}^{d})\to\mathcal{S}'^*_{\dagger}\left(\mathbb{R}^{d}, \mathcal{S}'^*_{\dagger}(\mathbb{R}^{d})\right)$, where $\iota(f)=\mathbf{f}$ is given by (\ref{eq:8}). Consider the restriction of $\iota$ to $\mathcal{D}'^*_{E'_{\ast}}$,
\begin{equation}
\label{embedding}
\iota:\mathcal{D}'^*_{E'_{\ast}}\to\mathcal{S}'^*_{\dagger}(\mathbb{R}^{d},E'_*),
\end{equation}
(for $f\in \DD'^*_{E'_*}$, the range of $\iota(f)$ is subset of $E'_*$ by Theorem \ref{karak}). Let $B_1$ be an arbitrary bounded subset of $\mathcal{S}^*_{\dagger}(\mathbb{R}^d)$. The set $B=\{\psi*\varphi|\,\varphi \in B_1,\|\psi\|_E\leq 1\}$ is bounded in $\mathcal{D}^*_{E}$ (by Lemma \ref{lemma:ime2}). For  $f\in \mathcal{D}'^*_{E'_{\ast}}$,
\beqs
\sup_{\varphi\in B_1} \|\langle\mathbf{f},\varphi\rangle\|_{E'}= \sup_{\varphi\in B_1} \|f*\check{\varphi}\|_{E'}=\sup_{\varphi\in B_1}\sup_{\|\psi\|_E\leq1}|\langle f,\psi*\varphi\rangle|=\sup_{\chi\in B}|\langle f,\chi\rangle|.
\eeqs
Hence, $\iota(f)\in\mathcal{S}'^*_{\dagger}(\mathbb{R}^{d},E'_*)$ ($\SSS^*_{\dagger}(\RR^d)$ is bornological) and $\iota$ is continuous. Furthermore, Proposition \ref{com} tells us that $\iota(\mathcal{D}_{E'_{\ast}}'^*)$ is precisely the subspace of $\mathcal{S}'^*_{\dagger}(\mathbb{R}^{d},E'_{\ast})$ consisting of those $\mathbf{f}$ which commute with all translations in the sense of Proposition \ref{com}. Since the translations $T_{h}$ are continuous operators on $E'_{\ast}$, we actually obtain that the range $\iota(\mathcal{D}_{E'_{\ast}}'^*)$ is a closed subspace of $\mathcal{S}'^*_{\dagger}(\mathbb{R}^{d},E'_{\ast})$.

\begin{corollary}\label{karak1}
For $B\subseteq\SSS'^*_{\dagger}(\mathbb{R}^d)$ the equivalent conditions from Theorem \ref{karak} are equivalent to the following:
\begin{itemize}
\item[$(vi)$] $\iota(B)$ is a bounded subset of $\mathcal{S}'^*_{\dagger}(\mathbb{R}^{d},E')$ (or equivalently of $\mathcal{S}'^*_{\dagger}(\mathbb{R}^{d},E'_*)$)
\end{itemize}
\end{corollary}

\begin{proof} $(i)\Rightarrow (vi)$ and $(vi)\Rightarrow (ii)$ are trivial.
\end{proof}

\begin{corollary}\label{cor:sequence}
Let $\left\{f_{\lambda}\right\}_{\lambda\in\Lambda}\subseteq\mathcal{D}_{E'_{\ast}}'^*$ be a bounded net (or similarly, a sequence). The following statements are equivalent:
\begin{itemize}
\item [$(i)$] $\left\{f_{\lambda}\right\}_{\lambda\in\Lambda}$ is convergent in $\mathcal{D}_{E'_{\ast}}'^*$.
\item [$(ii)$] $\left\{\iota(f_{\lambda})\right\}_{\lambda\in\Lambda}$ is convergent in $\mathcal{S}'^*_{\dagger}(\mathbb{R}^{d},E')$ (or equivalently in $\mathcal{S}'^*_{\dagger}(\mathbb{R}^{d},E'_{\ast})$).
\item [$(iii)$] There exists a convergent bounded net $\{g_{\lambda}\}_{\lambda\in\Lambda}$ in $E'$ and an ultradifferential operator $P(D)$ of class $*$ such that each $f_{\lambda}=P(D) g_{\lambda}$.
\item [$(iv)$] There exists a net $\{g_{\lambda}\}_{\lambda\in\Lambda}\subseteq E'_*\cap  UC_{\omega}$ which is convergent and bounded in $E'_*$ and in $ UC_{\omega}$ and an ultradifferential operator $P(D)$ of class $*$ such that $f_{\lambda}=P(D) g_{\lambda}$; if $E$ is reflexive one may choose $\{g_{\lambda}\}_{\lambda\in\Lambda}\subseteq E'\cap C_{\omega}$.
\end{itemize}
\end{corollary}

\begin{proof} We consider the Roumieu case as the Beurling case is similar. Let $(ii)$ hold. Since the image of $\DD'^{\{M_p\}}_{E'_*}$ under $\iota$ is closed subspace of $\mathcal{S}'^{\{M_p\}}_{\{A_p\}}(\mathbb{R}^{d},E'_{\ast})$, $\iota(f_{\lambda})\rightarrow \iota(f)$, for $f\in \DD'^{\{M_p\}}_{E'_*}$. As $B=\{\iota(f)\}\cup\{\iota(f_{\lambda})|\,\lambda\in \Lambda\}$ is bounded in $\mathcal{L}_b\left(\SSS^{\{M_p\}}_{\{A_p\}}(\RR^d), E'_*\right)$, it is equicontinuous ($\SSS^{\{M_p\}}_{\{A_p\}}(\RR^d)$ is barreled) and thus, there exists $(r_p)\in\mathfrak{R}$ such that the elements of $B$ can be extended to a bounded subset $\tilde{B}=\{\widetilde{\iota(f)}\}\cup\{\widetilde{\iota(f_{\lambda})}|\,\lambda\in \Lambda\}$ of $\mathcal{L}_b\left(X_{(r_p)}, E'_*\right)$. Moreover, $\widetilde{\iota(f_{\lambda})}(\varphi)\rightarrow \widetilde{\iota(f)}(\varphi)$ for each $\varphi$ in the dense subset $\SSS^{\{M_p\}}_{\{A_p\}}(\RR^d)$ of $X_{(r_p)}$. Since $\tilde{B}$ is bounded in $\mathcal{L}_b\left(X_{(r_p)}, E'_*\right)$, $\widetilde{\iota(f_{\lambda})}\rightarrow \widetilde{\iota(f)}$ in $\mathcal{L}_{\sigma}\left(X_{(r_p)}, E'_*\right)$, the Banach-Steinhaus theorem implies that it is also bounded in $\mathcal{L}_p\left(X_{(r_p)}, E'_*\right)$. Pick now, $(r'_p)\in\mathfrak{R}$ with $(r'_p)\leq (r_p)$ such that the inclusion $X_{(r'_p)}\rightarrow X_{(r_p)}$ is compact. Then the inclusion $\mathcal{L}_p\left(X_{(r_p)}, E'_*\right)\rightarrow \mathcal{L}_b\left(X_{(r'_p)}, E'_*\right)$ is continuous. Thus $\widetilde{\iota(f_{\lambda})}\rightarrow \widetilde{\iota(f)}$ in $\mathcal{L}_b\left(X_{(r'_p)}, E'_*\right)$. Now one can use a similar technique as in the proof of $(ii)\Rightarrow (iv)$ of Theorem \ref{karak} to conclude $(iii)$ and similar technique as in the proof of $(ii)\Rightarrow (v)$ of Theorem \ref{karak} to conclude $(iv)$. The implications $(i)\Rightarrow (ii)$, $(iii)\Rightarrow (i)$ and $(iv)\Rightarrow (i)$ are obvious.
\end{proof}

This corollary implies that the restriction of $\iota$ on each bounded subset $B$ of $\DD'^*_{E'_*}$ is topological homeomorphism between $B$ and $\iota(B)$.

For the proof of the following two results we refer to \cite[Proposition 6.7, Proposition 6.8]{DPPV}
\begin{theorem}\label{1517}
The spaces $\DD^{\{M_p\}}_E$ and $\tilde{\DD}^{\{M_p\}}_E$ are isomorphic as l.c.s..
\end{theorem}

\begin{proposition}\label{prop:reflexive}
If $E$ is reflexive, then $\mathcal{D}^{(M_p)}_{E}$ and $\mathcal{D}'^{\{M_p\}}_{E'}$ are $(FS^*)$-spaces, $\mathcal{D}^{\{M_p\}}_{E}$ and $\mathcal{D}'^{(M_p)}_{E}$ are $(DFS^*)$-spaces. Consequently, they are reflexive. In addition, $\mathcal{S}^*_{\dagger}(\mathbb{R}^{d})$ is dense in $\mathcal{D}'^*_{E'}$.
\end{proposition}

\subsection{Weighted  $\mathcal{D}'^{\ast}_{L^{p}_{\eta}}$ spaces}
\label{examples}
In this subsection we discuss some important examples of the spaces $\mathcal{D}^*_{E}$ and $\mathcal{D}'^*_{E'_{\ast}}$, where $E$ is taken as a weighted $L^{p}_{\eta}$ spaces. In the next considerations we retain the notation exactly as in Example \ref{example}. In particular, $\eta$ is ultrapolynomially bounded weight of class $\dagger$ and the number $q$ always stands for $p^{-1}+q^{-1}=1$ ($p\in[1,\infty]$). It should be mentioned that in the case $\eta=1$ and $M_p=A_p$ satisfying $(M.3)$ the spaces we study below were considered in \cite{Pilipovic} (see also \cite{PilipovicK}). The non-quasianalytic case with $A_{p}=M_{p}$ for general weights $\eta$ was studied in detail in \cite{DPPV}.

Consider now the spaces $\mathcal{D}^*_{L^{p}_{\eta}}$ for $p\in[1,\infty]$ and $\tilde{\DD}^{\{M_p\}}_{L^{\infty}_{\eta}}$ defined as in Section \ref{new spaces} by taking $E=L^{p}_{\eta}$. We also treat $\mathcal{D}_{C_{\eta}}$ defined via $E=C_{\eta}$. Once again, the case $p=\infty$ is an exception since $\mathcal{S}^*_{\dagger}(\mathbb{R}^{d})$ is not dense in $\mathcal{D}^*_{L^{\infty}_{\eta}}$ nor in $\tilde{\DD}^{\{M_p\}}_{L^{\infty}_{\eta}}$. Nonetheless, we can repeat the proof of Lemma \ref{regular} to prove that $\DD^{\{M_p\}}_{L^{\infty}_{\eta}}$ is regular and complete. Also, similarly as in Lemmas \ref{ultradiff} and \ref{udozc}, one can obtain that each ultradifferential operator of $*$ class acts continuously on $\DD^*_{L^{\infty}_{\eta}}$ and each ultradifferential operator of $\{M_p\}$ class acts continuously on $\tilde{\DD}^{\{M_p\}}_{L^{\infty}_{\eta}}$. Obviously $\DD^{\{M_p\}}_{L^{\infty}_{\eta}}$ is continuously injected into $\tilde{\DD}^{\{M_p\}}_{L^{\infty}_{\eta}}$ and by using \cite[Lemma 3.4]{Komatsu3} and employing a similar technique as in the proof of Lemma \ref{regular}, one can prove that this inclusion is in fact surjective. We denote by $\mathcal{B}^*_{\eta}$ the space $\DD^*_{L^{\infty}_{\eta}}$ and by $\dot{\mathcal{B}}^*_{\eta}$ the closure of $\SSS^*_{\dagger}(\RR^d)$ in $\mathcal{B}^*_{\eta}$. We denote by $\dot{\tilde{\mathcal{B}}}^{\{M_p\}}_{\eta}$ the closure of $\SSS^{\{M_p\}}_{\{A_p\}}(\RR^d)$ in $\tilde{\DD}^{\{M_p\}}_{L^{\infty}_\eta}$. We immediately see that $\dot{\mathcal{B}}^{(M_p)}_{\eta}=\mathcal{D}^{(M_p)}_{C_{\eta}}$.
In the Roumieu case this is result is given by the following theorem. Its proof is analogous to that of \cite[Theorem 7.2]{DPPV} and we omit it.

\begin{theorem}\label{edn}
The spaces $\mathcal{D}^{\{M_p\}}_{C_{\eta}}$, $\dot{\mathcal{B}}^{\{M_p\}}_{\eta}$ and $\dot{\tilde{\mathcal{B}}}^{\{M_p\}}_{\eta}$ are isomorphic one to another as l.c.s..
\end{theorem}

Proposition \ref{smooth prop} together with the estimate (\ref{bounds}) (resp. Proposition \ref{smooth prop} together with (\ref{boundss})) imply $\mathcal{D}^*_{L^{p}_{\eta}}\hookrightarrow\dot{\mathcal{B}}^*_{\check{\omega}_{\eta}}$ for every $p\in[1,\infty)$. It follows from Proposition \ref{prop:reflexive} that $\mathcal{D}^{*}_{L^p_\eta}$ is reflexive when $p\in(1,\infty)$.

 In accordance to Subsection \ref{subsection DE}, the weighted spaces $\mathcal{D}'^*_{L_{\eta}^p}$ are defined as
$\mathcal{D}'^*_{L_{\eta}^p}= (\mathcal{D}^*_{L_{\eta^{-1}}^q})'$ where $p^{-1}+q^{-1}=1$ if $p\in (1,\infty)$; if $p=1$,
$\mathcal{D}'^*_{L_{\eta}^1}=(\mathcal{D}^*_{C_{\eta}})'=(\dot{\mathcal{B}}^*_\eta)'$ and for $p=\infty$ we define $\mathcal{D}'^*_{L^{\infty}_{\eta}}:=\mathcal{D}'^*_{UC_{\eta}}=(\DD^*_{L^1_{\eta}})'$. We write
$\mathcal{B}'^*_{\eta}=\mathcal{D}'^*_{L^{\infty}_{\eta}}$ and $\dot{\mathcal{B}}'^*_{\eta}$ for the closure of
$\mathcal{S}^*_{\dagger}(\mathbb{R}^{d})$ in $\mathcal{B}'^*_{\eta}$.

For $f\in\DD'^*_{L^1_{\eta}}$, by Theorem \ref{karak}, there exist an ultradifferential operator $P(D)$ of class $*$ and $g\in L^1_{\eta}$ such that $f=P(D)g$. But, since $\SSS^*_{\dagger}(\RR^d)$ is dense in $L^1_{\eta}$, there exists a sequence $g_n\in\SSS^*_{\dagger}(\RR^d)$, $n\in\ZZ_+$, such that $g_n\rightarrow g$ in $L^1_{\eta}$. Hence $\SSS^*_{\dagger}(\RR^d)\ni f_n=P(D)g_n\rightarrow f$ in $\DD'^*_{L^1_{\eta}}$, i.e., $\SSS^*_{\dagger}(\RR^d)$ is sequentially dense in $\DD'^*_{L^1_{\eta}}$. Moreover, as an easy consequence of the Sobolev imbedding theorem, we obtain that $\DD^*_{L^p_{\eta}}$ is continuously injected into $C^{\infty}(\RR^d)$ for each $p\in[1,\infty)$. Since $\SSS^*_{\dagger}(\RR^d)\hookrightarrow\SSS(\RR^d)\hookrightarrow C^{\infty}(\RR^d)$, $\DD^*_{L^p_{\eta}}$ is dense in $C^{\infty}(\RR^d)$, hence $\EE'(\RR^d)$ is continuously injected into $\DD'^*_{L^q_{1/\eta}}$. In particular the delta (ultra)distribution belongs to $\DD'^*_{L^q_{1/\eta}}$.

\begin{theorem}\label{theorBB}
The strong bidual of $\dot{\mathcal{B}}^*_{\eta}$ is isomorphic to $\DD^*_{L^{\infty}_{\eta}}$ as l.c.s.. In the Roumieu case $\DD^{\{M_p\}}_{L^{\infty}_{\eta}}$ and $\tilde{\DD}^{\{M_p\}}_{L^{\infty}_{\eta}}$ are isomorphic l.c.s.. Moreover $\dot{\mathcal{B}}^{(M_p)}_{\eta}$ is a distinguished $(F)$-space and consequently $\DD'^{(M_p)}_{L^1_{\eta}}$ is barreled and bornological.
\end{theorem}

\begin{proof} We may assume that $\eta$ is continuous (cf. Example \ref{example}). Let $\tilde{\eta}(x)=1/\left(\eta(x)\langle x\rangle^{d+1}\right)$. Then, clearly $\tilde{\eta}(x)$ is a continuous ultrapolynomially bounded weight of class $\dagger$ and $\dot{\mathcal{B}}^*_{\eta}\hookrightarrow \DD^*_{L^2_{\tilde{\eta}}}$. Since $\SSS^*_{\dagger}(\RR^d)$ is dense in $\DD'^*_{L^1_{\eta}}$, we have $\DD'^*_{L^2_{1/\tilde{\eta}}}\hookrightarrow\DD'^*_{L^1_{\eta}}$. This, together with Proposition \ref{prop:reflexive}, implies that $\left(\DD'^*_{L^1_{\eta}}\right)'_b$ (where $b$ stands for the strong topology) is continuously injected into $\DD^*_{L^2_{\tilde{\eta}}}$, i.e., the elements of $\left(\DD'^*_{L^1_{\eta}}\right)'_b$ are smooth functions. In the Roumieu case, we already saw that $\DD^{\{M_p\}}_{L^{\infty}_{\eta}}$ and $\tilde{\DD}^{\{M_p\}}_{L^{\infty}_{\eta}}$ are equal as sets. First we prove that the bidual of $\dot{\mathcal{B}}^*_{\eta}$ is isomorphic to $\DD^{(M_p)}_{L^{\infty}_{\eta}}$, and to $\tilde{\DD}^{\{M_p\}}_{L^{\infty}_{\eta}}$ respectively. Let $r>0$ (let $(r_p)\in\mathfrak{R}$ and set $R_{\alpha}=\prod_{j=1}^{|\alpha|}r_j$). Consider the set
\beqs
&{}&B_r=\left\{\frac{(\eta(a))^{-1}r^{|\alpha|}D^{\alpha}\delta_a}{M_{\alpha}}\Big|\, a\in\RR^d,\, \alpha\in\NN^d\right\}
\\
&{}&
\left(\mbox{resp.}\,\,B_{(r_p)}=\left\{\frac{(\eta(a))^{-1}D^{\alpha}\delta_a}{M_{\alpha}R_{\alpha}}\Big|\, a\in\RR^d,\, \alpha\in\NN^d\right\}\right).
\eeqs
One easily proves that $B_r$ is a bounded subset of $\DD'^{(M_p)}_{L^1_{\eta}}$ ($B_{(r_p)}$ is a bounded subset of $\DD'^{\{M_p\}}_{L^1_{\eta}}$). Hence if $\psi\in \left(\DD'^*_{L^1_{\eta}}\right)'_{b}$, $\psi(B_r)$ (resp. $\psi(B_{(r_p)})$) is bounded in $\CC$ and thus
\beqs
&{}&\sup_{a,\alpha}\frac{\left|(\eta(a))^{-1}r^{|\alpha|}D^{\alpha}\psi(a)\right|}{M_{\alpha}}=\sup_{f\in B_r}\left|\langle \psi, f\rangle\right|<\infty\\
&{}&\left(\mbox{resp.}\,\,\sup_{a,\alpha}\frac{\left|(\eta(a))^{-1}D^{\alpha}\psi(a)\right|}{M_{\alpha}R_{\alpha}}=\sup_{f\in B_{(r_p)}}\left|\langle \psi, f\rangle\right|<\infty\right).
\eeqs
We obtain that $\left(\DD'^*_{L^1_{\eta}}\right)'\subseteq \DD^*_{L^{\infty}_{\eta}}$ and the inclusion $\left(\DD'^{(M_p)}_{L^1_{\eta}}\right)'_{b}\rightarrow \DD^{(M_p)}_{L^{\infty}_{\eta}}$,  ($\left(\DD'^{\{M_p\}}_{L^1_{\eta}}\right)'_{b}\rightarrow \tilde{\DD}^{\{M_p\}}_{L^{\infty}_{\eta}}$) is continuous.

 Let $\psi\in \DD^*_{L^{\infty}_{\eta}}$. If $f\in \DD'^*_{L^1_{\eta}}$, by Theorem \ref{karak} there exist an ultradifferential operator $P(D)$ of class $*$ and $g\in L^1_{\eta}$ such that $f=P(D)g$. Define $S_{\psi}$ by
\beqs
S_{\psi}(f)=\int_{\RR^d}g(x)P(-D)\psi(x)dx.
\eeqs
Obviously, the integral on the right hand side is absolutely convergent. We will prove that $S_{\psi}$ is well defined element of $\left(\DD'^*_{L^1_{\eta}}\right)'$. Let $\tilde{P}(D)$, $\tilde{g}\in L^1_{\eta}$ be such that $f=\tilde{P}(D)\tilde{g}$. Let $\varphi_n\in\SSS^*_{\dagger}(\RR^d)$, $n\in\ZZ_+$, be as in $ii)$ of Lemma \ref{appincl}. Then it is easy to verify that
\beqs
\int_{\RR^d}P(-D)\left(\varphi_n(x)\psi(x)\right)g(x)dx&\rightarrow& \int_{\RR^d}P(-D)\psi(x)g(x)dx,\\
\int_{\RR^d}\tilde{P}(-D)\left(\varphi_n(x)\psi(x)\right)\tilde{g}(x)dx&\rightarrow& \int_{\RR^d}\tilde{P}(-D)\psi(x)\tilde{g}(x)dx,
\eeqs
as $n\rightarrow\infty$. Also, observe that for each $n\in\ZZ_+$, $\varphi_n\psi\in\SSS^*_{\dagger}(\RR^d)$ and thus
\beqs
\int_{\RR^d}P(-D)\left(\varphi_n(x)\psi(x)\right)g(x)dx&=&{}_{\SSS'^*_{\dagger}}\langle f,\varphi_n\psi\rangle_{\SSS^*_{\dagger}}\\
&=& \int_{\RR^d}\tilde{P}(-D)\left(\varphi_n(x)\psi(x)\right)\tilde{g}(x)dx.
\eeqs
Hence, $S_{\psi}$ is a well defined mapping $\DD'^{\{M_p\}}_{L^1_{\eta}}\rightarrow \CC$, since it does not depend on the representation of $f$. To prove that it is continuous we consider first the Beurling case. The space $\DD'^{(M_p)}_{L^1_{\eta}}$ is a complete $(DF)$-space. Thus it is enough to prove that the restriction of $S_{\psi}$ on each bounded subset of $\DD'^{(M_p)}_{L^1_{\eta}}$ is continuous (see the corollary to \cite[Theorem 6.7, p. 154]{Sch}), i.e., we have to prove that if $\{f_{\lambda}\}_{\lambda\in\Lambda}$ is bounded net which converges to $f$ in $\DD'^{(M_p)}_{L^1_{\eta}}$, then $S_{\psi}(f_{\lambda})\rightarrow S_{\psi}(f)$. If $\{f_{\lambda}\}_{\lambda\in\Lambda}$ is such net, Corollary \ref{cor:sequence} implies that there exists a net $\{g_{\lambda}\}_{\lambda\in\Lambda}\subseteq L^1_{\eta}$ which is bounded and convergent in $L^1_{\eta}$ and an ultradifferential operator $P(D)$ of class $(M_p)$ such that $f_{\lambda}=P(D)g_{\lambda}$ and $f=P(D)g$ where $g\in L^1_{\eta}$ is the limit of $\{g_{\lambda}\}_{\lambda\in\Lambda}$. But then one easily verifies that $g_{\lambda}P(-D)\psi\rightarrow gP(-D)\psi$ in $L^1$, hence $S_{\psi}(f_{\lambda})\rightarrow S_{\psi}(f)$. Thus $S_{\psi}\in \left(\DD'^{(M_p)}_{L^1_{\eta}}\right)'$. In the Roumieu case, as $\DD'^{\{M_p\}}_{L^1_{\eta}}$ is an $(F)$-space one can similarly prove that $S_{\psi}\in \left(\DD'^{\{M_p\}}_{L^1_{\eta}}\right)'$. We obtain that $\left(\DD'^{(M_p)}_{L^1_{\eta}}\right)'=\DD^{(M_p)}_{L^{\infty}_{\eta}}$ (resp. $\left(\DD'^{\{M_p\}}_{L^1_{\eta}}\right)'=\tilde{\DD}^{\{M_p\}}_{L^{\infty}_{\eta}}$) as sets and $\left(\DD'^*_{L^1_{\eta}}\right)'_{b}$ has stronger topology than the latter. In the Beurling case, $\left(\DD'^{(M_p)}_{L^1_{\eta}}\right)'_b$ is an $(F)$-space as the strong dual of the $(DF)$-space $\DD'^{(M_p)}_{L^1_{\eta}}$. Hence the open mapping theorem proves that $\left(\DD'^{(M_p)}_{L^1_{\eta}}\right)'_b=\DD^{(M_p)}_{L^{\infty}_{\eta}}$ as l.c.s.. In the Roumieu case, let $V=B^{\circ}$ be a neighborhood of zero $\left(\DD'^{\{M_p\}}_{L^1_{\eta}}\right)'_{b}$ for $B$ a bounded subset of $\DD'^{\{M_p\}}_{L^1_{\eta}}$. By Theorem \ref{karak}, there exist an ultradifferential operator $P(D)$ of class $\{M_p\}$ and bounded subset $B_1$ of $L^1_{\eta}$ such that each $f\in B$ can be represented by $f=P(D)g$ for some $g\in B_1$. There exists $C_1\geq 1$ such that $\|g\|_{L^1_{\eta}}\leq C_1$ for all $g\in B_1$. Also, since $P(D)=\sum_{\alpha}c_{\alpha}D^{\alpha}$ is of $\{M_p\}$ class, there exist $(r_p)\in\mathfrak{R}$ and $C_2\geq 1$ such that $\left|c_{\alpha}\right|\leq C_2/(M_{\alpha}R_{\alpha})$ (see the proof of Lemma \ref{udozc}). Observe the neighborhood of zero $\ds W=\left\{\psi\in \tilde{\DD}^{\{M_p\}}_{L^{\infty}_{\eta}}\Big|\, \sup_{x,\alpha}\frac{(\eta(x))^{-1}\left|D^{\alpha}\psi(x)\right|}{M_{\alpha}\prod_{j=1}^{|\alpha|}(r_j/2)}\leq \frac{1}{2C_1C_2C_3}\right\}$ in $\tilde{\DD}^{\{M_p\}}_{L^{\infty}_{\eta}}$, where we put $C_3=\sum_{\alpha}2^{-|\alpha|}$. One easily verifies that $W\subseteq V$. We obtain that $\left(\DD'^{\{M_p\}}_{L^1_{\eta}}\right)'_{b}$ and $\tilde{\DD}^{\{M_p\}}_{L^{\infty}_{\eta}}$ are isomorphic l.c.s.. Hence $\tilde{\DD}^{\{M_p\}}_{L^{\infty}_{\eta}}$ is a complete $(DF)$-space (since $\DD'^{\{M_p\}}_{L^1_{\eta}}$ is an $(F)$-space). As the identity mapping $\DD^{\{M_p\}}_{L^{\infty}_{\eta}}\rightarrow\tilde{\DD}^{\{M_p\}}_{L^{\infty}_{\eta}}$ is continuous and bijective, it remains to prove that the inverse is continuous. Since $\tilde{\DD}^{\{M_p\}}_{L^{\infty}_{\eta}}$ is a $(DF)$-space, to prove the continuity of the inverse mapping it is enough to prove that its restriction to every bounded subset of $\tilde{\DD}^{\{M_p\}}_{L^{\infty}_{\eta}}$ is continuous (see the corollary to \cite[Theorem 6.7, p. 154]{Sch}). If $B$ is a bounded subset of $\tilde{\DD}^{\{M_p\}}_{L^{\infty}_{\eta}}$ then for every $(r_p)\in\mathfrak{R}$, $\ds \sup_{\psi\in B} \sup_{\alpha}\frac{\left\|D^{\alpha}\psi\right\|_{L^{\infty}_{\eta}(\RR^d)}}{M_{\alpha}R_{\alpha}}<\infty$. Hence, by \cite[Lemma 3.4]{Komatsu3}, there exists $h>0$ such that $\ds \sup_{\psi\in B} \sup_{\alpha}\frac{h^{|\alpha|}\left\|D^{\alpha}\psi\right\|_{L^{\infty}_{\eta}(\RR^d)}}{M_{\alpha}}<\infty$, i.e., $B$ is bounded in $\DD^{\{M_p\}}_{L^{\infty}_{\eta}}$. Since every bounded subset of $\DD^{\{M_p\}}_{L^{\infty}_{\eta}}$ is obviously bounded in $\tilde{\DD}^{\{M_p\}}_{L^{\infty}_{\eta}}$, $\DD^{\{M_p\}}_{L^{\infty}_{\eta}}$ and $\tilde{\DD}^{\{M_p\}}_{L^{\infty}_{\eta}}$ have the same bounded sets. Let $\psi_{\lambda}$ be a bounded net in $\tilde{\DD}^{\{M_p\}}_{L^{\infty}_{\eta}}$ which converges to $\psi$ in $\tilde{\DD}^{\{M_p\}}_{L^{\infty}_{\eta}}$. Then there exist $0<h\leq 1$ and $C>0$ such that
\beqs
\sup_{\lambda} \sup_{\alpha}\frac{h^{|\alpha|}\left\|D^{\alpha}\psi_{\lambda}\right\|_{L^{\infty}_{\eta}}}{M_{\alpha}}\leq C \mbox{ and } \sup_{\alpha}\frac{h^{|\alpha|}\left\|D^{\alpha}\psi\right\|_{L^{\infty}_{\eta}}}{M_{\alpha}}\leq C.
\eeqs
Choose $0<h_1<h$. Let $\varepsilon>0$ be arbitrary but fixed. Take $p_0\in\ZZ_+$ such that $(h_1/h)^{|\alpha|}\leq \varepsilon/(2C)$ for all $|\alpha|\geq p_0$. Since $\psi_{\lambda}\rightarrow \psi$ in $\tilde{\DD}^{\{M_p\}}_{L^{\infty}_{\eta}}$, for the sequence $r_p=p$, $p\in\ZZ_+$, there exists $\lambda_0$ such that for all $\lambda\geq \lambda_0$ we have $\ds \sup_{\alpha}\frac{\left\|D^{\alpha}\left(\psi_{\lambda}-\psi\right)\right\|_{L^{\infty}_{\eta}}}{M_{\alpha}R_{\alpha}}\leq \frac{\varepsilon}{p_0!}$. Then for $|\alpha|<p_0$, we have $\ds \frac{h_1^{|\alpha|}\left\|D^{\alpha}\left(\psi_{\lambda}-\psi\right)\right\|_{L^{\infty}_{\eta}}}{M_{\alpha}}\leq \varepsilon$. For $|\alpha|\geq p_0$, we have
\beqs
\frac{h_1^{|\alpha|}\left\| D^{\alpha}\left(\psi_{\lambda}-\psi\right)\right\|_{L^{\infty}_{\eta}}}{M_{\alpha}}\leq 2C\left(\frac{h_1}{h}\right)^{|\alpha|}\leq \varepsilon.
\eeqs
It follows that $\psi_{\lambda}\rightarrow \psi$ in $\DD^{\{M_p\},h_1}_{L^{\infty}_{\eta}}$ and hence in $\DD^{\{M_p\}}_{L^{\infty}_{\eta}}$. We obtain that the induced topology by $\tilde{\DD}^{\{M_p\}}_{L^{\infty}_{\eta}}$ on every bounded subset of $\tilde{\DD}^{\{M_p\}}_{L^{\infty}_{\eta}}$ is stronger than the induced topology by $\DD^{\{M_p\}}_{L^{\infty}_{\eta}}$. Hence the identity mapping $\tilde{\DD}^{\{M_p\}}_{L^{\infty}_{\eta}}\rightarrow \DD^{\{M_p\}}_{L^{\infty}_{\eta}}$ is continuous.

It remains to prove that $\dot{\mathcal{B}}^{(M_p)}_{\eta}$ is distinguished. Denote by $\DD^{(M_p)}_{L^{\infty}_{\eta},\sigma}$ the space $\DD^{(M_p)}_{L^{\infty}_{\eta}}$ equipped with the weak topology from the duality $\left\langle \DD'^{(M_p)}_{L^1_{\eta}}, \DD^{(M_p)}_{L^{\infty}_{\eta}}\right\rangle$. We have to prove that each bounded subset of $\DD^{(M_p)}_{L^{\infty}_{\eta}}$ (the strong bidual of $\dot{\mathcal{B}}^{(M_p)}_{\eta}$) is contained in the closure of a bounded subset of $\dot{\mathcal{B}}^{(M_p)}_{\eta}$ in $\DD^{(M_p)}_{L^{\infty}_{\eta},\sigma}$. Let $B$ be a bounded subset of $\DD^{(M_p)}_{L^{\infty}_{\eta}}$. Let $\varphi_n\in\SSS^{(M_p)}_{(A_p)}(\RR^d)$, $n\in \ZZ_+$, be the sequence from $ii)$ of Lemma \ref{appincl}. Then, $\varphi_n\psi\in\SSS^{(M_p)}_{(A_p)}(\RR^d)$ for each $n\in\ZZ_+$, $\psi\in B$. For $r>0$ one easily verifies that $\|\varphi_n\psi\|_{L^{\infty}_{\eta},r}\leq \|\varphi\|_{L^{\infty},2r}\|\psi\|_{L^{\infty}_{\eta},2r}$. Hence the set $\tilde{B}=\{\varphi_n\psi|\, n\in\ZZ_+,\psi\in B\}$ is bounded subset of $\dot{\mathcal{B}}^{(M_p)}_{\eta}$. Let $\psi\in B$ and $f\in \DD'^{(M_p)}_{L^1_{\eta}}$. By Theorem \ref{karak}, there exist an ultradifferential operator $P(D)$ of class $(M_p)$ and $g\in L^1_{\eta}$ such that $f=P(D)g$. Then one easily verifies that $gP(-D)(\varphi_n\psi)\rightarrow gP(-D)\psi$ in $L^1$, thus $\langle \varphi_n\psi, f\rangle \rightarrow \langle \psi,f\rangle$, i.e., $\varphi_n\psi\rightarrow \psi$ in $\DD^{(M_p)}_{L^{\infty}_{\eta},\sigma}$, which proves that $B$ belongs in the closure of $\tilde{B}$ in $\DD^{(M_p)}_{L^{\infty}_{\eta},\sigma}$.
\end{proof}

\subsection{Convolution and multiplication} Our previous work allows us to extend all results on convolution and multiplicative products on $\mathcal{D}'^*_{E'_{\ast}}$ from \cite{DPV,DPPV} to our spaces. We omit the proofs of the following propositions because they go in the same lines as those of \cite[Theorem 4 and Proposition 11]{DPV} (adapting them with the aid of our results from the previous subsections).

\begin{proposition}\label{prop4.13}
We have the (continuous) inclusions
$\mathcal{D}^*_{L^1_\omega}\hookrightarrow\mathcal{D}^*_E\hookrightarrow\dot{\mathcal{B}}^*_{\check{\omega}}$ and $\mathcal{D}'^*_{L^1_{\check{\omega}}}\rightarrow\mathcal{D}'^*_{E'_{\ast}}\rightarrow\mathcal{B}'^*_{\omega}$. If $E$ is reflexive, one has $\mathcal{D}'^*_{L^{1}_{\check{\omega}}}\hookrightarrow\mathcal{D}'^*_{E'}\hookrightarrow\dot{\mathcal{B}}'^*_{\omega}$.
\end{proposition}

In particular, we have $\mathcal{D}^*_{L^{1}_{\omega_{\eta}}}\hookrightarrow\mathcal{D}^*_{L^{p}_{\eta}}\hookrightarrow \dot{\mathcal{B}}^*_{\check{\omega}_{\eta}}$ and $\mathcal{D}'^*_{L^{1}_{\omega_{\eta}}}\hookrightarrow\mathcal{D}'^*_{L^{p}_{\eta}}\hookrightarrow \dot{\mathcal{B}}'^*_{\check{\omega}_{\eta}}$ for $1\leq p <\infty$ (for $p=1$ in the latter dense inclusion we have used the fact $\SSS^*_{\dagger}(\RR^d)\hookrightarrow \mathcal{D}'^*_{L^{1}_{\eta}}$). In addition, $\dot{\mathcal{B}}^*_{{\eta}}\hookrightarrow
\dot{\mathcal{B}}^*_{\omega_{\eta}}$ and $\dot{\mathcal{B}}'^*_{{\eta}}\hookrightarrow \dot{\mathcal{B}}'^*_{\omega_{\eta}}$.  We can now define multiplicative and convolution operations on $\mathcal{D}'^*_{E'_{\ast}}$. In the next proposition we denote by $\mathcal{O}'^*_{\dagger,C,b}$ the space $\mathcal{O}'^*_{\dagger,C}$ equipped with the strong topology from the duality $\left\langle \mathcal{O}^*_{\dagger,C},\mathcal{O}'^*_{\dagger,C}\right\rangle$.

\begin{proposition}\label{conv-prod}
The convolution mappings $\ast:\mathcal{D}'^*_{E'_{\ast}}\times\mathcal{D}'^*_{L^{1}_{\check{\omega}}}\to \mathcal{D}'^*_{E'_{\ast}}$ and  $
\ast:\mathcal{D}'^*_{E'_{\ast}}\times\mathcal{O}'^*_{\dagger,C,b}(\RR^d)\to \mathcal{D}'^*_{E'_{\ast}}
$ are continuous.
The convolution and multiplicative products are hypocontinuous in the following cases:
$
\cdot:\mathcal{D}'^*_{E'_{\ast}}\times\mathcal{D}^*_{L^{1}_{\omega}}\to \mathcal{D}'^*_{L^{1}},
$
$
\cdot:\mathcal{D}'^*_{L^{1}_{\check{\omega}}}\times\mathcal{D}^*_{E}\to \mathcal{D}'^*_{L^{1}}.
$ and
$
\ast:\mathcal{D}'^*_{E'_{\ast}}\times\mathcal{D}^*_{\check{E}}\to \mathcal{B}^*_{\omega}
$.
When $E$ is reflexive, we have $
\ast:\mathcal{D}'^*_{E'}\times\mathcal{D}^*_{\check{E}}\to \dot{\mathcal{B}^*_{\omega}}
$.
\end{proposition}


\begin{thebibliography}{99}

\bibitem{PilipovicK} R.~Carmichael, A.~Kami\'nski, S.~Pilipovi\'c, \emph{Boundary values and convolution in ultradistribution spaces,} World Scientific Publishing Co. Pte. Ltd., Hackensack, NJ, 2007.

\bibitem{DPV} P.~Dimovski, S.~Pilipovi\'c, J.~Vindas, \emph{New distribution spaces associated to translation-invariant Banach spaces,} Monatsh. Math., in press (doi:10.1007/s00605-014-0706-3).

\bibitem{DPVBV} P.~Dimovski, S.~Pilipovi\'c, J.~Vindas, \emph{Boundary values of holomorphic functions and heat kernel method in translation-invariant distribution spaces,} Complex Var. Elliptic Equ., in press (doi:10.1080/17476933.2014.1002399).

\bibitem{DPPV} P. Dimovski, S. Pilipovi\'{c}, B. Prangoski,  J. Vindas, \emph{Convolution of ultradistributions and ultradistribution spaces associated to translation-invariant Banach spaces,} Kyoto J. Math., to appear.

\bibitem{hilleph} E.~Hille, R.~S.~Phillips, \emph{Functional analysis and semi-groups,} American Mathematical Society, Providence, R.I, 1974.

\bibitem{kisynski} J.~Kisy\'{n}ski, \emph{On Cohen's proof of the factorization theorem,} Ann. Polon. Math. \textbf{75} (2000), 177--192.

\bibitem{Komatsu1} H.~Komatsu, \emph{Ultradistributions, I: Structure theorems and a characterization,} J. Fac. Sci. Univ. Tokyo Sect. IA Math. \textbf{20} (1973), 25--105.

\bibitem{Komatsu3} H.~Komatsu, \emph{Ultradistributions, III: Vector valued ultradistributions and the theory of kernels,} J. Fac. Sci. Univ. Tokyo Sect. IA Math. \textbf{29} (1982), 653--717.



\bibitem{Pilipovic} S.~Pilipovi\'c, \emph{Characterizations of bounded sets in spaces of ultradistributions,} Proc. Amer. Math. Soc. \textbf{120} (1994), 1191--1206.

\bibitem{ppv} S.~Pilipovi\'c, B.~Prangoski, J.~Vindas, \emph{On quasianalytic classes of Gelfand-Shilov type. Parametrix and convolution,} preprint.

\bibitem{BojanL} B.~Prangoski, \emph{Laplace transform in spaces of ultradistributions,} Filomat \textbf{27} (2013), 747--760.

\bibitem{Sch} H.~H.~Schaefer, \emph{Topological vector spaces,} Springer-Verlag, New York-Heidelberg-Berlin, 1970.

\bibitem{SchwartzV} L.~Schwartz, \emph{Th\'eorie des distributions \`{a} valeurs vectorielles. I,} Ann. Inst. Fourier \textbf{7} (1957), 1--141.

\bibitem{S}
    L. Schwartz,
    \emph{Th\'{e}orie des distributions,}
    Hermann, Paris, 1966.

\end{thebibliography}
\end{document}